%% file: global.tex
\documentclass[12pt,reqno,a4paper]{amsart}
\usepackage{amssymb,amsmath,amscd,amstext,amsthm,amsfonts}
\usepackage{graphicx}
\usepackage{epstopdf}
\usepackage[mathscr]{eucal}
\usepackage{epsfig,color}
\usepackage{subfig}

\usepackage[ansinew]{inputenc} 

\numberwithin{equation}{section}

\newtheorem{theorem}{Theorem}[section]
\newtheorem{proposition}[theorem]{Proposition}

\newtheorem{lemma}[theorem]{Lemma}

{\theoremstyle{definition}
{
\newtheorem{remark}[theorem]{Remark}

}}

{\theoremstyle{definition}
\newtheorem{defn}[theorem]{Definition}}

\newcommand{\len}[1]{\ell\left({#1}\right)}

\newcommand{\cal}{\mathcal}

\newcommand{\LL}{{\cal L}}

\newcommand{\Nn}{{\mathbb{N}}}

\newcommand{\Rr}{{\mathbb{R}}}

\newcommand{\Zz}{{\mathbb{Z}}}

\def\dist{\operatorname{dist}}

\def\supp{\operatorname{supp}}

\def\Leb{\operatorname{Leb}}

\newcommand{\abs}[1]{\left| #1\right|}
\newcommand{\norm}[1]{\left\| #1\right\|}

\newcommand{\comment}[1]{}

\begin{document}

\title[Hyperbolic polygonal billiards]{Hyperbolic polygonal billiards with finitely many ergodic SRB measures}

\date{\today}

\author[Del Magno]{Gianluigi Del Magno}
\address{Universidade Federal da Bahia, Instituto de Matem\'atica\\
Avenida Adhemar de Barros, Ondina \\
40170--110 Salvador, BA, Brasil}
\email{gdelmagno@ufba.br}

\author[Lopes Dias]{Jo\~ao Lopes Dias}
\address{Departamento de Matem\'atica and CEMAPRE, ISEG\\
Universidade de Lisboa\\
Rua do Quelhas 6, 1200--781 Lisboa, Portugal}
\email{jldias@iseg.ulisboa.pt}

\author[Duarte]{Pedro Duarte}
\address{Departamento de Matem\'atica and CMAF \\
Faculdade de Ci\^encias\\
Universidade de Lisboa\\
Campo Grande, Edificio C6, Piso 2\\
1749--016 Lisboa, Portugal 
}
\email{pmduarte@fc.ul.pt}

\author[Gaiv\~ao]{Jos\'e Pedro Gaiv\~ao}
\address{Departamento de Matem\'atica and CEMAPRE, ISEG\\
Universidade de Lisboa\\
Rua do Quelhas 6, 1200--781 Lisboa, Portugal}
\email{jpgaivao@iseg.ulisboa.pt}

\begin{abstract}
We study 
polygonal billiards with reflection laws contracting the reflected angle towards the normal. It is shown that if a polygon does not have parallel sides facing each other, then the corresponding billiard map has finitely many ergodic SRB measures whose basins cover a set of full Lebesgue measure. 
\end{abstract}

\maketitle

\section{Introduction}
\label{sec:introduction}
\input{intro.tex}

\section{Billiard map} 
\label{sec:billiard map}

\input{billiardmap.tex}

\section{Growth of the branching number}
\label{sec:branching}
\input{complexity.tex}

\section{Existence of SRB measures}
\label{sec:SRB}
\input{srb.tex}

\section{Basins of the ergodic SRB measures}
\label{sec:Basins}
\input{basins.tex}


\bibliographystyle{plain}

\end{document}

%% file: intro.tex
The study of the asymptotic behaviour of billiards is an important subject in the theory of dynamical systems. Billiards exhibit a rich variety of statistical properties depending on the geometry of their tables and the reflection law considered.

In this work, we are interested in 
polygonal 
tables. The billiard map of a polygonal billiard with the standard reflection law (the angle of reflection equals the angle of incidence) 
is conservative and non-chaotic: it
preserves a measure that is absolutely continuous with respect to the Lebesgue measure, and
all its Lyapunov exponents are equal to zero.

A completely different dynamics arises when 
the reflection law is contracting~\cite{markarian10}, 
i.e. when the reflection angle measured from the normal is a contraction of the incidence angle.
In this case, the billiard is a dissipative system: its map 
does not longer preserve an absolutely continuous measure, and may have
attractors~\cite{arroyo09,arroyo12,MDDGP12}. Indeed, if there are no period two orbits, then the map has 
a uniformly hyperbolic 
attractor~\cite{MDDGP13}.
Notice that period two orbits correspond to collisions perpendicular to a pair of parallel sides of the billiard table. These orbits are parabolic, and their union forms an attractor.

For billiards in generic convex polygons with a strong contracting reflection law, 
we have recently proved the existence of countably many ergodic Sinai-Ruelle-Bowen measures (SRB), each one supported on a uniformly hyperbolic attractor~\cite{MDDGP13}. This result is significantly extended in the current paper
by enlarging the class of allowed polygons, including now non-convex polygons, and more importantly by removing any restriction 
on the contraction factor of the reflection law (cf.~\cite{MDDGP12,MDDGP13,MDDGP14,MDDGP15}). 
In addition, we establish that the basins of the ergodic SRB measures cover a set of full Lebesgue measure. The full result is the following.

\begin{theorem}\label{main thm}
For every polygon without parallel sides facing each other and every contracting reflection law, the corresponding billiard map has a hyperbolic  attractor supporting finitely many ergodic SRB measures whose basins cover a set of full Lebesgue measure. Every ergodic SRB measure admits a decomposition into finitely many Bernoulli components, each component having exponential decay of correlations for H\"older observables.
Every SRB measure of the billiard map is a convex combination of the ergodic SRB measures. Finally, the set of periodic points is dense in the attractor.
\end{theorem}

Besides its specific interest, the previous result may prove useful in studying polygonal billiards with the standard reflection law, because they lie at the boundary of the class of billiards considered in this paper. 

A long standing conjecture of J.~Palis ~\cite{Palis} states there exists a dense set of dynamical systems such that each of them have finitely many attractors with ergodic SRB measures whose basins of attraction cover a set with full Lebesgue measure. Since polygons without parallel sides facing each other are dense in the space of all polygons, Theorem~\ref{main thm} verifies this conjecture for 
polygonal billiards with contracting reflection laws.

Polygonal billiards are piecewise smooth systems: they have discontinuities corresponding to trajectories reaching a corner of the table. Discontinuities represent an obstacle to hyperbolicity in that they may prevent the system from having local invariant manifolds, or local invariant manifolds of uniform size.
The local fractioning of the unstable manifolds produced by the discontinuities is measured by the branching number of the singular sets (see section~\ref{sec:branching}). The control of the growth of this number is key to guarantee that the expansion along the unstable direction prevails over the fractioning of local unstable manifolds caused by the discontinuities, and
allows us to extend the results of~\cite{MDDGP13}. Our proof relies on results for general hyperbolic piecewise smooth maps of Pesin~\cite{pe} and Sataev~\cite{Sataev92}.

The rest of the paper is organized as follows. In Section~\ref{sec:billiard map} we introduce the billiard maps and other basic definitions.
The growth of the branching number is studied in Section~\ref{sec:branching}. This result allows us to establish the existence of finitely many ergodic SRB measures in Section~\ref{sec:SRB}.
Finally, in Section~\ref{sec:Basins}, we show that the basins of 
these measures cover the entire phase space up to a set of zero Lebesgue measure.

%% file: billiardmap.tex
Let $P$ be a polygonal domain (open and connected set) of $ \Rr^{2} $ with $d$-sides and perimeter equal to one. The billiard in $ P $ with the specular reflection law is the flow on the unit tangent bundle of $ P $ generated by the motion of a free point-particle in $ P $ with specular reflection at $ \partial P $, i.e., the angle of reflection equals the angle of incidence. The corresponding billiard map $ \Phi_{P} $ is the first return map on $ M $, the set of unit vectors attached to $ \partial P $ and pointing inside $ P $.

Each element of $ x \in M $ can be identified with a pair $ (s,\theta) $, where $ s $ is the arclength parameter of $ \partial P $ of the base point of $ x $, and $ \theta $ is the angle formed by $ x $ with the positively oriented tangent to $ \partial P $ at $ s $. Accordingly, we can write
$$ 
M = S^1 \times [-\pi/2,\pi/2].
$$

The domain of $ \Phi_{P} $ does not coincide with the entire $ M $. To specify it, we first introduce the sets $ V $ and $ S $. Let $ 0 = s_{1} < \cdots < s_{d} < 1 $ be the values of the arc-length parameter corresponding to the vertices of $ P $. The set $ V $ is given by 
$$ 
V = \{s_{1},\ldots,s_{d}\} \times (-\pi/2,\pi/2),
$$ 
whereas the set $ S $ is the subset of $ M $ consisting of elements whose forward trajectory hit a vertex of $ P $ at the first collision with $ \partial P $.
Define
$$
N = V \cup S \qquad \text{and} \qquad N^{+} = N \cup \partial M.
$$ 
Both sets $ N $ and $ N^{+} $ consist of finitely many smooth curves 
~\cite[Proposition~2.1]{MDDGP13}.

The map $ \Phi_{P} $ is defined on $ M \setminus N^+ $, and is a piecewise smooth map 
with singular set $ N^{+} $ in the sense of Definition~\ref{def:piecewise map}. Observe that $ \Phi_{P}(x) $ is the unit vector corresponding to the first collision of the trajectory of $ x \in M \setminus N^{+}$ with $ \partial P $. 
For a detailed definition of $ \Phi_{P} $, we refer the reader to~\cite{MDDGP13} for polygonal table, and to \cite{CM} for more general tables. 

A \emph{reflection law} is a function $ f \colon (-\pi/2,\pi/2) \to (-\pi/2,\pi/2) $. For example, the specular reflection law corresponds to the identity function $ f(\theta)=\theta $. Let $ R_{f} \colon M \to M $ be the map $ R_{f}(s,\theta)=(s,f(\theta)) $. The billiard map for the polygon $ P $ with reflection law $ f $ is the map $\Phi_{f,P} \colon M \setminus N^+ \to M $ given by
\[
\Phi_{f,P} = R_{f} \circ \Phi_{P}.
\]
This map is just the first return map on $ M $ of the billiard flow in the polygon $ P $ with refection law $ f $ (see Figure~\ref{fig:poly}). We call $\Phi_{f,P}$ the \textit{billiard map of $P$ with reflection law $f$}.

\begin{figure}
  \begin{center}
  \includegraphics[width=2in]{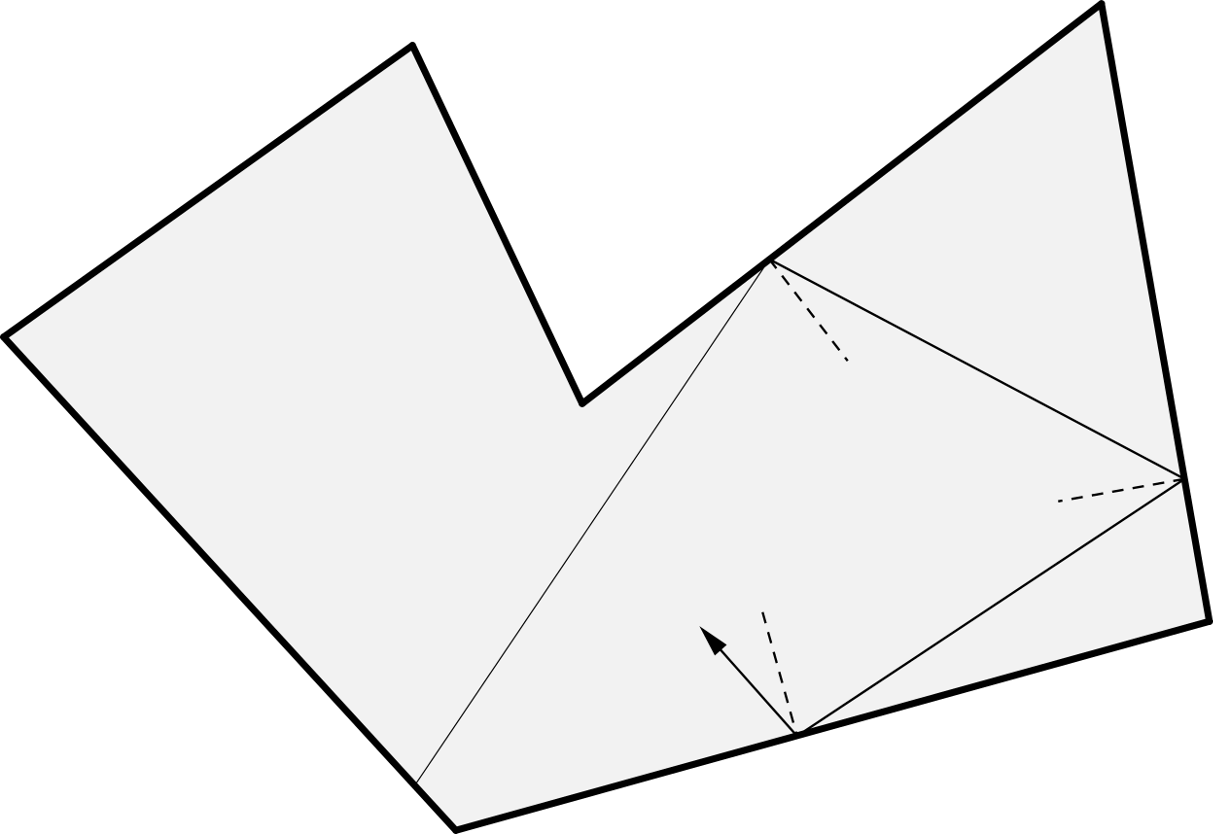}
  \end{center}
  \caption{Billiard flow with a contracting reflection law in a polygon without parallel sides facing each other.}
  \label{fig:poly}
\end{figure}

If $ f $ is differentiable, then so is $ \Phi_{f,P} $. In this case, the explicit expression for the derivative of $ \Phi_{f,P} $ can be easily computed (for the derivative of $ \Phi_{P} $, see~\cite{CM}): 
\begin{equation} 
\label{eq:derivative}
D\Phi_{f,P}{(s,\theta)}
=-\left( 
\begin{array}{cc}
\dfrac{\cos \theta}{\cos \bar{\theta}_{1}} & \dfrac{t(s,\theta)}{\cos \bar{\theta}_{1}} \\ [1em]
0 & f'(\bar{\theta}_{1}) 
\end{array}
\right), 
\end{equation}
where $ (\bar{s}_{1},\bar{\theta}_{1}) = \Phi_{P}(s,\theta) $, $ (s_{1},\theta_{1}) = \Phi_{f,P}(s,\theta) $, and $t(s,\theta)$ denotes the Euclidean distance in $ \Rr^{2} $ between the points of $\partial P$ with coordinates $s$ and $s_{1}$. 

Given a differentiable reflection law $ f $, we define 
\[ 
\lambda(f) = \sup_{\theta \in (-\pi/2,\pi/2)} |f'(\theta)|.
\] 
A differentiable reflection law is called \emph{contracting} if $\lambda(f) < 1$. The simplest example of a contracting reflection law is $ f(\theta) = \sigma \theta $ with $0 < \sigma < 1$~\cite{arroyo09,arroyo12,MDDGP12,markarian10}.

{\bf Standing assumptions on $ f $:} we assume throughout the paper that $ f $ satisfies the following condition: 
\begin{enumerate}
\item $ f $ is a $ C^{2} $ embedding from $ [-\pi/2,\pi/2] $ to $ [-\pi/2,\pi/2] $,
\item $ f $ is contracting,
\item $ f(0)=0 $,
\item $ f'(\theta)>0 $ for $ \theta \in (-\pi/2,\pi/2) $.
\end{enumerate}

Since $ f'>0 $, all the entries of $ D\Phi_{f,P} $ have the same sign. This simple fact will play an important role in the arguments presented in the next section.

%% file: complexity.tex
To simplify our notation, from now on we shall write $\Phi$ instead of  $ \Phi_{f,P}$. For $ n \ge 1 $, define 
\[ 
N^{+}_{n} = N^{+} \cup \Phi^{-1}(N^{+}) \cup \cdots \cup \Phi^{-n+1}(N^{+}),
\] 
and
\[ 
S^{+}_{n} = S \cup \Phi^{-1}(S) \cup \cdots \cup \Phi^{-n+1}(S).
\]           
The set $ N^{+}_{n} $ contains all the points of $ M $ where the map $ \Phi^{n} $ is not defined. Since $ N = V \cup S $ and $ \Phi^{-1}(V \cup \partial M) = \emptyset $, it follows that
\[
N^{+}_{n} = V \cup \partial M \cup S^{+}_{n}.
\] 
As a direct consequence of the definition of $ S^{+}_{n} $, we have 
\[ 
S^{+}_{n+k} = S^{+}_{n} \cup \Phi^{-n}(S^{+}_{k}) \qquad \text{for } k > 0.
\]
Hence, $ N^{+}_{n+k} \setminus N^{+}_{n} = S^{+}_{n+k} \setminus S^{+}_{n} \subset \Phi^{-n}(S^{+}_{k}) $ and
\begin{equation}
\label{eq:image}
\Phi^{n}(N^{+}_{n+k} \setminus N^{+}_{n}) \subset S^{+}_{k}.
\end{equation}

\begin{defn} 
\label{de:sector}
A \emph{sector of order $ n $ with vertex} 
$ x \in N^{+}_{n} $ is a connected component $ \Delta $ of $ U \setminus N^{+}_{n} $ with $ U \subset M $ being an open ball centered at a point $ x \in N^{+}_{n} $ such that the closure of $ U $ intersects only smooth components of $ N^{+}_{n} $ meeting at $ x $ (see Figure~\ref{fig:sector}).

The smooth curves forming the boundary of $ \Delta $ and meeting at $ x $ are called the \emph{boundary curves} of $ \Delta $. A sector $ \Delta' \subset \Delta $ of order greater than $ n $ with vertex $ x $ is called a \emph{sub-sector} of $ \Delta $.
\end{defn}

\begin{figure}
  \begin{center}
  \includegraphics[width=2in]{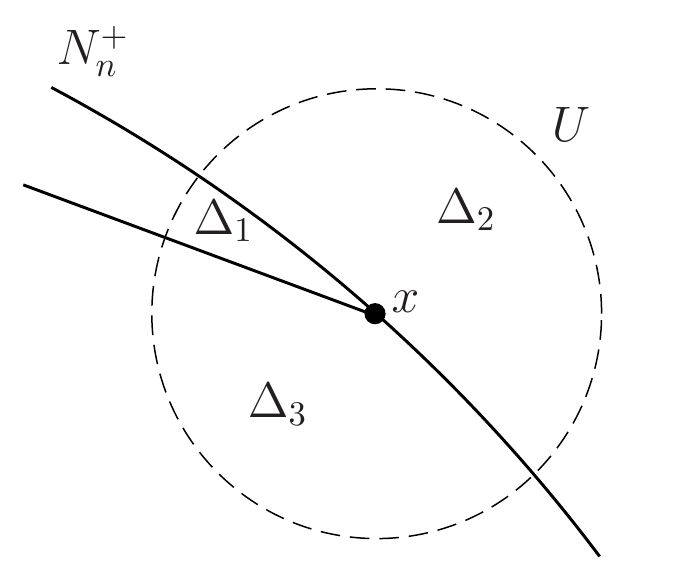}
  \end{center}
  \caption{Sectors $\Delta_1$, $\Delta_2$ and $\Delta_3$ sharing the same vertex $x\in N_n^+$.}
  \label{fig:sector}
\end{figure}
 
From the definition above, it follows that $ \Phi^{n} \colon \Delta \to \Phi^{n}(\Delta) $ is a $ C^{2} $ diffeomorphism for every sector $ \Delta $ of order $ n $. In particular, the first $ n $th iterates of points of $ \Delta $ visit the same sequence of sides of $ P $. Moreover, if $ \Gamma $ is a boundary curve of $ \Delta $ such that $ \Gamma \not \subset \partial M $, then either $ \Gamma \subset V $ or $ \Phi^{k-1}(\Gamma) \subset S^{+}_{1} $ for some $ 1 \le k \le n $. In other words, all the trajectories starting at points of a boundary curve not contained in $ \partial M $ either emerge from the same vertex or hit the same vertex at the $ k $th collision with $ \partial P $.

\begin{defn}
A $ C^{1} $ curve $ t \mapsto \gamma(t) = (s(t),\theta(t)) \in M $ is called \emph{increasing} if $ s'(t)\theta'(t) \ge 0 $ for every $ t $, and is called \emph{strictly increasing} if $ s'(t)\theta'(t) > 0 $ for every $ t $. A \emph{decreasing} curve and a \emph{strictly decreasing} curve are defined similarly by requiring that $ s'(t)\theta'(t) \le 0 $ for every $ t $ and $ s'(t)\theta'(t) < 0 $ for every $ t $, respectively. A curve $ t \mapsto \gamma(t) = (s(t),\theta(t)) \in M $ is called \emph{horizontal} if there exists a constant $ c $ such that $ \theta(t) = c $ for every $ t $.
\end{defn}

\begin{remark}
\label{re:monotonicity}	
The set $ V \cup S^{+}_{n} $ consists of finitely many strictly decreasing $ C^{2} $ curves~\cite[Proposition~2.3]{MDDGP15}. 
\end{remark}

\begin{lemma}
	\label{le:boundary}
	We have $ \overline{S^{+}_{n}} \cap \partial M = \emptyset $ for $ n \in \Nn $.
\end{lemma}

\begin{proof}
It can be easily checked that the case $ n=1 $ holds true. Hence, in the rest of the proof, we can limit ourselves to consider the case when $ n \ge 2 $, and $ S^{+}_{n} $ is replaced by $ R_{n} := S^{+}_{n} \setminus S^{+}_{1} $. We argue by contradiction, and suppose that $ \overline{R_{n}} \cap \partial M \neq \emptyset $ for some $ n \in \Nn $. Let $ x \in \overline{R_{n}} \cap \partial M $. There are precisely two sectors of order $ n $ with vertex $ x $ that have a boundary curve lying on $ \partial M $. Denote the one on the right-hand side of $ x $ by $ \Delta $. One boundary curve of $ \Delta $ is contained in $ \partial M $, whereas the other is contained in $ R_{n} $.  

Now, choose a point $ y $ in the interior the boundary curve $ \gamma $ of $ \Delta $ that is not contained in $ \partial M $ in such a way that the vertical segment $ \gamma_{1} $ with one endpoint in $ y $ and the other one on $ \partial M $ is contained in the closure of $ \Delta $. Denote by $ \gamma_{2} $ be the sub-curve of $ \gamma $ with endpoints $ x $ and $ y $, and denote by $ z $ the endpoint of the segment $ \gamma_{1} $ belonging to $ \partial M $.

From the definition of $ S^{+}_{n} $, it follows that $ \gamma'_{2} :=\Phi(\gamma_{2}) $ is contained in a smooth component of $ S^{+}_{k} $ for some $ 1 \le k < n $. Since the smooth components of $ S^{+}_{n} $ are strictly decreasing curves by Remark~\ref{re:monotonicity}, 
so is $ \gamma'_{2} $. The map $ \Phi $ is $ C^{2} $ differentiable on $ \Delta $, and so $ \gamma'_{1} := \Phi(\gamma_{1}) $ is a $ C^{2} $ curve. Since $ \gamma_{1} $ is vertical, and the entries of $ D \Phi $ have the same sign (see~\eqref{eq:derivative}), it follows that $ \gamma'_{1} $ is strictly increasing. 

Since $ y \in \gamma_{1} \cap \gamma_{2} $, it follows that $ \Phi(y)  \in \gamma'_{1} \cap \gamma'_{2} $. Moreover, since $ \gamma'_{1} $ and $ \gamma'_{2} $ are strictly monotone, both limits $ \lim_{\gamma_{1} \ni w \to z} \Phi(w) $ and $ \lim_{\gamma_{2} \ni w \to x} \Phi(w) $ exist. It is easy to see that these limits coincide, because $ x,z \in \partial M = S^{1} \times \{-\pi/2,\pi/2\} $. In conclusion, the curves $ \gamma'_{1} $ and $ \gamma'_{2} $ intersect at both their endpoints. However, this is impossible, because one curve is strictly increasing and the other is strictly decreasing. 
\end{proof}

\begin{defn}
A sector $ \Delta $ of order $ n $ is called \emph{regular} if $ \Phi^{n}|_{\Delta} $ admits an extension $ \Psi_{\Delta} \colon \overline{\Delta} \to \Psi_{\Delta}(\overline{\Delta}) $ that is a $ C^{2} $ diffeomorphism, where $ \overline{\Delta} $ is the closure of $ \Delta $.
\end{defn}

\begin{lemma}
\label{le:increasing}
Suppose that $ \Delta $ is a regular sector of order $ n $. 
If $ \, \Gamma $ is a boundary curve of $ \Delta $ that is not contained in $ \partial M $, then $ \Psi_{\Delta}(\Gamma) $ is a $ C^{2} $ increasing curve.
\end{lemma}

\begin{proof}
Define $ F(s,\theta) = (s,f(-\theta)) $ for $ (s,\theta) \in M $, and let $ N^{-} = F(V \cup S^{+}_{1}) $. Remark~\ref{re:monotonicity} combined with $ f'>0 $ implies that $ N^{-} $ consists of strictly increasing $ C^{2} $ curves. 

Let $ \Gamma $ be as in the statement of the lemma. By the remark after Definition~\ref{de:sector}, we have either $ \Gamma \subset V $ or $ \Phi^{i}(\Gamma) \subset S^{+}_{1} $ for some $ 0 \le i \le n-1 $. Then, it is not difficult to see that there exists $ 0 \le k \le n-1 $ such that $ \Psi_{\Delta}(\Gamma) \subset \Phi^{k}(N^{-}) $. Since $ N^{-} $ consists of strictly increasing $ C^{2} $ curves, to obtain the wanted conclusion is enough to observe that if $ \gamma $ is a strictly increasing curve such that $ \gamma \cap N^{+}_{i} = \emptyset $, then $ \Phi^{i}(\gamma) $ is strictly increasing curve as well. This is so, because all the entries of the matrix of $ D \Phi^{i} $ have the same sign (see~\eqref{eq:derivative}). 
\end{proof}

\begin{defn}
Let $ \Delta $ be a sector with vertex $ x $. A curve $ \gamma \colon [a,b] \to M $ is called \emph{$ \Delta $--curve} if $ \gamma(a) = x $ and $ \gamma(t) \in \Delta $ for $ t \in (a,b] $. 
\end{defn}

In the next lemma, we give a necessary condition for a regular sector $ \Delta $ to contain singular $ \Delta $--curves. By definition, if $ \Delta $ is a sector of order $ n $, then $ \Delta \cap S^{+}_{n} = \emptyset $. Hence, if there exists a $ \Delta $--curve $ \gamma $ contained in $ S^{+}_{n+k} $ for some $ k>0 $, then we must have $ \gamma \subset S^{+}_{n+k} \setminus S^{+}_{n} $.

\begin{lemma}
\label{le:necessary}
Suppose that $ \Delta $ is a regular sector of order $ n $. If there exists a $ \Delta $--curve contained in $ S^{+}_{n+k} $ for some $ k>0 $, then there exists also a horizontal $ \Delta $--curve.
\end{lemma}

\begin{proof}
First of all, we observe that if $ x $ is the vertex of $ \Delta $, then $ x \notin \partial M $ by Lemma~\ref{le:boundary}. 
Now, let $ \gamma $ be the $ \Delta $--curve contained in $ S^{+}_{n+k} $. Then $ \gamma \subset S^{+}_{n+k} \setminus S^{+}_{n} $, and so~\eqref{eq:image} implies that $ \Psi_{\Delta}(\gamma) \subset S^{+}_{k} $. Hence $ \Phi^{n}_{\Delta}(\gamma) $ 
is a strictly decreasing curve passing through $ \Psi_{\Delta}(x) $. Next, the region $ \Psi_{\Delta}(\Delta) $ is bounded by the curves $ \Gamma_{1} $ and $ \Gamma_{2} $, which are the images under $ \Psi_{\Delta} $ of the boundary curves of $ \Delta $, passing through $ \Psi_{\Delta}(x) $, and are increasing by Lemma~\ref{le:increasing}.  

Summarizing, $ \Psi_{\Delta}(\Delta) $ is bounded by the increasing curves $ \Gamma_{1} $ and $ \Gamma_{2} $, and contains a strictly decreasing curve $ \Psi_{\Delta}(\gamma) $. Moreover, all these curves pass through $ \Psi_{\Delta}(x) $. It is now easy to see that $ \Psi_{\Delta}(\Delta) $ must contain a horizontal curve $ \chi $ passing through $ \Psi_{\Delta}(x) $. Since horizontal curves are mapped by $ \Psi^{-1}_{\Delta} $ into horizontal curves, we conclude that $ \Psi^{-1}_{\Delta}(\chi) $ is a horizontal $ \Delta $--curve.
\end{proof}

Since sectors are bounded by strictly decreasing curves, among the sectors of a given order with a given vertex, there are only two of them whose closure contains horizontal curves passing through the vertex. 

\begin{defn}
	Let $ b_{n} $ be the maximum number of distinct regular sectors of order $ n $ with the same vertex. We call $ b_{n} $ the \emph{branching number of order $ n $}.
\end{defn}

\begin{theorem}
\label{th:linear}
We have $ b_{n} \le (2n-1)b_{1} $ for $ n \ge 1 $.
\end{theorem}

\begin{proof}
To prove the corollary, we show that $ b_{n+1} - b_{n} \le 2b_{1} $ for $ n \ge 1 $, which implies the wanted conclusion. This is achieved by estimating the maximum number of regular sectors of order $ n $ containing components of $ S^{+}_{n+1} $ passing through their vertices, and the maximum number of such components. 

Consider a regular sector $ \Delta $ of order $ n $ with vertex $ x $. Let $ a $ be the maximum number of $ \Delta $--curves contained in $ S^{+}_{n+1} $. Since $ \Delta $ is of order $ n $, these curves are indeed contained in $ S^{+}_{n+1} \setminus S^{+}_{n} $. By~\eqref{eq:image}, their image under $ \Psi_{\Delta} $ consists of an equal number of curves contained in $ S^{+}_{1} $ and meeting at $ \Psi_{\Delta}(x) $. Hence $ a \le b_{1} $. If $ a > 0 $, then Lemma~\ref{le:necessary} implies that there exist horizontal $ \Delta $--curves. Since there are only two sectors of order $ n $ with vertex $ x $ having this property, we conclude that $ b_{n+1} - b_{n} \le 2b_{1} $.
\end{proof}

\begin{defn}
A non-regular sector is called \emph{primary} if it is not a sub-sector of a non-regular sector.
\end{defn}

\begin{remark}
\label{re:non-regular}	
It is not difficult to see that if $ \Delta $ is a non-regular sector of order $ n $ with vertex $ x $, then either $ x \in \partial M $, or there exists $ 0 \le k < n $ such that $ \Phi^{k} $ is a $ C^{2} $ diffeomorphism on the closure of $ \Delta $, and $ \Phi^{k}(x) $ is a tangential singularity. Moreover, every sub-sector of $ \Delta $ is non-regular as well. If $ \Delta $ is primary, then $ k = n-1 $. 
\end{remark}

%% file: srb.tex
In this section, we prove our main result. It relies on results of Pesin and Sataev on the existence and properties of SRB measures for general hyperbolic piecewise smooth maps. For the convenience of the reader, we first state the results with the necessary definitions.

\subsection{Hyperbolic piecewise smooth maps}
\label{se:hyperbolic}

Let $ \mathcal{M} $ be a 
smooth manifold 
with Riemannian metric $ \rho $. The Lebesgue measure of $ \mathcal{M} $ generated by $ \rho $ is denoted by $ \nu $. Let $ \mathcal{U} $ be a connected open subset of $ \mathcal{M} $ with compact closure. Finally, let $\mathcal{N}$ be a subset of $\mathcal{U}$ closed in the relative topology of $ \mathcal{U} $.

\begin{defn}\label{def:piecewise map}
A map $\mathcal{F} \colon \mathcal{U} \setminus \mathcal{N} \to \mathcal{M}$ is called a \emph{piecewise smooth map} if $\mathcal{F}$ is a $C^2$ diffeomorphism from $\mathcal{U}\setminus \mathcal{N}$ onto its image $\mathcal{F}(\mathcal{U}\setminus \mathcal{N})$. The set $\mathcal{N}^+:=\mathcal{N}\cup\partial \mathcal{M}$ is called the \emph{singular set} of $ \mathcal{F} $.
\end{defn}

Let $\mathcal{N}^-=\partial (\mathcal{F}(\mathcal{M}\setminus \mathcal{N}^+))$.

\begin{defn}
Let $ \mathcal{U}^{+} = \{x \in \mathcal{U} \colon \mathcal{F}^{n}(x) \not \in \mathcal{N}^{+} \,\, \forall n \ge 0\} $ be the set of all elements of $ M $ with infinite positive semi-orbit. Define
\[
\mathcal{D}={\bigcap_{n \ge 0} \mathcal{F}^{n}(\mathcal{U}^{+})}.
\]
The set $\mathcal{A}=\overline{\mathcal{D}}$ is called the \emph{generalized attractor} of $ \mathcal{F} $.
\end{defn}

\noindent
{\bf Condition A1:} There exist positive constants $ \widetilde A $ and $ a $ such that for every $ x \in \mathcal{U}\setminus \mathcal{N}^{+} $, 
\begin{equation*}
\label{cdn H2}
\|D^2 \mathcal{F}(x)\| \le \widetilde A \rho(x,\mathcal{N}^+)^{-a} \quad \text{and} \quad \|D^2 \mathcal{F}^{-1}(\mathcal{F}(x))\| \le \widetilde A \rho(x,\mathcal{N}^-)^{-a}.
\end{equation*}

A cone in $ T_{x} \mathcal{M} $, $ x \in \mathcal{U} $ with an axial linear subspace $ P \subset T_{x} \mathcal{M} $ and angle $ \alpha>0 $ is the set given by
\[
C_\alpha(x,P) = \{v \in T_{x} \mathcal{M} \colon \angle (v,P) \le \alpha \}.
\] 

\medskip
\noindent
{\bf Condition A2:} The map $ \mathcal{F} $ is uniformly hyperbolic. Namely, there exist two constants $ c > 0 $, $ \lambda> 1 $ and cones 
$$
C^{s}(x)=C_{\alpha^s(x)}(x,P^s(x))\quad\text{and}\quad C^{u}(x)=C_{\alpha^u(x)}(x,P^u(x))\,,
$$
with axial subspaces $ P^{u}(x), P^{s}(x) $ and positive angles $ \alpha^{u}(x), \alpha^{s}(x) $ for $ x \in \mathcal{U} \setminus \mathcal{N}^{+} $ such that 
\begin{enumerate}
\item $ T_{x} \, \mathcal{U} = P^{u}(x) \oplus P^{s}(x) $, 
\item $ \dim P^{u}(x) $ and $ \dim P^{s}(x) $ are constant,
\item the angle between $ C^{u}(x) $ and $ C^{s}(x) $ is uniformly bounded away from zero, 
\item $ D\mathcal{F}(x)(C^{u}(x)) \subset C^{u}(\mathcal{F}(x)) $ and $ D\mathcal{F}^{-1}(\mathcal{F}(x))(C^{s}(\mathcal{F}(x))) \subset C^{s}(x) $ for $ x \in \mathcal{U} \setminus \mathcal{N}^{+} $,
\item if $ x \in \mathcal{U}^{+} $ and $ n \in \Nn $, then $ \|D\mathcal{F}^{n}(x)v\| \ge c \lambda^{n} \|v\| $ for $ v \in C^{u}(x) $, and $ \|D\mathcal{F}^{-n}(\mathcal{F}^{n}(x))v\| \ge c \lambda^{n} \|v\| $ for $ v \in C^{s}(x) $. 
\end{enumerate}
 
\begin{defn}
We say that a piecewise smooth map $ \mathcal{F} $ is \emph{hyperbolic} or that $ \mathcal{A} $ is a \emph{hyperbolic attractor} if $ \mathcal{F} $ satisfies Conditions~A1 and~A2. 
\end{defn}

Following \cite{pe}, for every $ \epsilon>0 $ and every $ l \in\Nn $, we define
\begin{align*}
\hat{\mathcal{D}}^{+}_{\epsilon,l} &= \left\{x \in \mathcal{U}^+ \colon \rho(\mathcal{F}^{ n}(x),\mathcal{N}^+) \ge l^{-1} e^{-n \epsilon} \,\,\, \forall n \ge 0 \right\},\\
\mathcal{D}^{-}_{\epsilon,l} &= \left\{x \in \mathcal{A} \colon \rho(\mathcal{F}^{- n}(x),\mathcal{N}^-) \ge l^{-1} e^{-n \epsilon} \,\,\, \forall n \ge 0 \right\},\\
\mathcal{D}^{+}_{\epsilon,l} &= \hat{\mathcal{D}}^{+}_{\epsilon,l}\cap \mathcal{A},
\end{align*}
and
\[
\mathcal{D}^{\pm}_{\epsilon} = \bigcup_{l \geq1} \mathcal{D}^{\pm}_{\epsilon,l}, \qquad \mathcal{D}^{0}_{\epsilon} = \mathcal{D}^{-}_{\epsilon} \cap \mathcal{D}^{+}_{\epsilon}.
\]

Roughly speaking, the sets $\mathcal{D}^{+}_{\epsilon,l}$ (resp. $ \mathcal{D}^{-}_{\epsilon,l} $) consists of points in $\mathcal{D}$ whose forward (resp. backward) orbit does not get too close to the singular set. These sets are compact.

The local stable manifold $W^s_{loc}(x)$ through a point $x\in \mathcal{U}^+$ is the set of all $y\in \mathcal{U}^+$ such that $\mathcal{F}^{n}(x)$ and $\mathcal{F}^{n}(y)$ belong to the same connected component of $ \mathcal{M} \setminus \mathcal{N}^{+} $ for every $n\geq 0$, and $ \rho(\mathcal{F}^{n}(x), \mathcal{F}^{n}(y)) \to 0 $ as $ n \to +\infty $.
The local unstable manifold $W^u_{loc}(x)$ at $ x \in \mathcal{D} $ 
is defined similarly by replacing $ \mathcal{F} $ with $ \mathcal{F}^{-1} $.

Local stable and unstable manifolds exist at each point of $\mathcal{D}_\epsilon^0$ for $ \epsilon $ sufficiently small.

\begin{proposition}\label{prop:stable manifolds}
If $ \mathcal{F} $ is a hyperbolic piecewise smooth map, then there exists $\epsilon_0>0$ such that  for all $0<\epsilon<\epsilon_0$,
$ W^{s}_{loc}(x) $ (resp. $ W^{u}_{loc}(x) $) 
is a $C^1$ embedded submanifold of uniform size $ \delta_{l}>0 $ for every $l\in\Nn$ and for every $ x \in \hat{\mathcal{D}}^{+}_{\epsilon,l} $ (resp. $ x \in \mathcal{D}^{-}_{\epsilon,l} $). 
\end{proposition}

\begin{proof}
This is \cite[Proposition 4]{pe} for points in ${\mathcal{D}}^{+}_{\epsilon,l} $ (resp. $ \mathcal{D}^{-}_{\epsilon,l} $). 
In order to obtain the same result for points in $\hat{\mathcal{D}}^{+}_{\epsilon,l} $ it is enough to observe that the local stable manifold $ W^{s}_{loc}(x) $ depends only on the forward orbit of $x$.
\end{proof}

\begin{defn}
\label{de:srb}
Let $ \mathcal{F} $ be a hyperbolic piecewise smooth map. An invariant Borel probability measure $ \mu $ on the attractor $ \mathcal{A} $ is called\footnote{These measures are called \emph{Gibbs u-measures} in~\cite{pe}.} 
SRB 
if $ \mu(\mathcal{D}^{0}_{\epsilon}) = 1 $ with $\epsilon>0$ as in Proposition~\ref{prop:stable manifolds}, and the conditional measures of $ \mu $ on the local unstable manifolds 
are absolutely continuous.
\end{defn}

\begin{defn}
Given an invariant Borel probability $\mu$ on $ \mathcal{M} $, 
its \emph{basin} $B(\mu)$ is the set of points $x\in \mathcal{U}^+$ such that 
$$ 
\lim_{n\to +\infty}\frac{1}{n}\,\sum_{j=0}^{n-1}
\varphi(\mathcal{F}^j(x)) = \int \varphi\,d\mu
$$
for any continuous function $\varphi \colon \mathcal{M}\to\Rr$. We say that $\mu$ is a {\em physical measure} if $B(\mu)$ has positive $\nu$-measure. 
\end{defn}

\begin{proposition}[{\cite[Theorem 3]{pe}}]
\label{thm  physical measures}
Every ergodic SRB measure is a physical measure.
\end{proposition}

Let $ \mathcal{N}_{\epsilon}^+ \subset \mathcal{M} $ be the $ \epsilon $-neighbourhood of $ \mathcal{N}^+ $ for $ \epsilon>0 $. The Lebesgue measure of a submanifold $ W \subset \mathcal{M} $ is denoted by $ \nu_{W} $.  A smooth submanifold $ W \subset \mathcal{M} $ is called a \emph{$ u $-manifold} if the dimension of $ W $ is equal to the dimension of the unstable subspaces of $ \mathcal{F} $, and the tangent space of $ W $ at $ x $ is contained in $ C^{u}(x) $ for every $ x \in W $. 

\medskip
\noindent
{\bf Condition H:} There exist constants $ C>0 $, $ a \in (0,1) $, $ \beta>0 $ and $ \epsilon_{0}>0 $ such that for every $ u $-manifold $ W $, every $ n \ge 1 $ and every $ \epsilon \in (0,\epsilon_{0}) $,
\[
\nu_{W}\left(W \cap \mathcal{F}^{-n}(\mathcal{N}_{\epsilon}^+)\right) \le C \epsilon^{\beta} \left(a^{n} + \nu_{W}(W)\right).
\]

Roughly speaking, this condition states that the relative measure of points in a $u$-manifold ending up in a small neighbourhood of $\mathcal{N}^+$ is of the same order as the size of the neighbourhood. To the best of our knowledge, Condition~H was first introduced~{\cite{CZ} (see also~\cite{CM}).

Theorem~\ref{th:sataev} below contains the main results of Pesin and Sataev concerning SRB measures for hyperbolic piecewise smooth maps. 
It states that if such a map $ \mathcal{F} $ satisfies Condition~H, then $ \mathcal{F} $ admits SRB measures, each of them being a convex combinations of finitely many ergodic SRB measures~\cite[Theorems~5.12 and 5.15]{Sataev92}. Moreover, every ergodic SRB measure decomposes into finitely many Bernoulli components cyclically permuted by $ \mathcal{F} $~\cite[Theorem~4]{pe}, and the periodic orbits of $ \mathcal{F} $ are dense in the attractor $ \mathcal{A} $~\cite[Theorem~11]{pe}. 

\begin{theorem}
\label{th:sataev}
Let $ \mathcal{F} $ be a hyperbolic piecewise smooth map satisfying Condition~H. There exist finitely many ergodic SRB measures $ \mu_{1},\ldots,\mu_{m} $ concentrated on pairwise disjoint subsets $ \mathcal{E}_{1},\ldots,\mathcal{E}_{m} $ of the attractor $ \mathcal{A} $ such that 
\begin{enumerate}
\item for every SRB measure $ \mu $, there exist $ \alpha_{1},\ldots,\alpha_{m} \ge 0 $ with $ \sum^{m}_{i=1} \alpha_{i} = 1 $ such that $ \mu = \sum^{m}_{i} \alpha_{i} \mu_{i} $, 
\item for each $ i = 1,\ldots,m $, there exist disjoint subsets $ \mathcal{M}_{i,1},\ldots,\mathcal{M}_{i,k_{i}} $ with $ k_{i} \in \Nn $ such that $ \mathcal{E}_{i} = \bigcup^{k_{i}}_{j=1} \mathcal{M}_{i,j} $ (mod 0), $ \mathcal{F}(\mathcal{M}_{i,j}) = \mathcal{M}_{i,j+1} $ for $ 1 \le j < k_{i} $, $ \mathcal{F}(\mathcal{M}_{i,k_{i}}) = \mathcal{M}_{i,1} $, and the system $ (\mathcal{F}^{k_{i}}|_{\mathcal{M}_{i,j}},\mu_{i,j}) $ with $ \mu_{i,j} $ being the normalized restriction of $ \mu_{i} $ to $ \mathcal{M}_{i,j} $ 
is Bernoulli,
\item the set of periodic points of $ \mathcal{F} $ is dense in $ \mathcal{A} $.
\end{enumerate}
\end{theorem}

The system $ (\mathcal{F}^{k_{i}}|_{\mathcal{M}_{i,j}},\mu_{i,j}) $ is called a \emph{Bernoulli component} of $ \mathcal{F} $.

We observe that Pesin obtained a weaker result than Conclusion~(1) of Theorem~\ref{th:sataev}. He proved the existence of SRB measures, and that each SRB measure is a convex combinations of countably many ergodic SRB measures.

Condition~H does not appear in the works of Pesin and Sataev. However, this condition is equivalent  to Conditions~H3 and~H4 assumed by Sataev. Pesin assumed similar but weaker conditions. Conditions~H3 and H4 of Sataev are the following.

\medskip
\noindent
{\bf Condition H3:} There exist positive constants $ B,\beta',\epsilon_{1} $ such that 
\[
\nu(\mathcal{F}^{-n}(\mathcal{N}^{+}_{\epsilon})) < B \epsilon^{\beta'} \qquad \text{for } n \ge 1 \text{ and } \epsilon \in (0,\epsilon_{1}).
\] 

A smooth submanifold $ W \subset M $ is a \emph{$ u $-manifold} if the dimension of $ W $ is equal to the one of the unstable subspaces of $ \mathcal{F} $, and the tangent space of $ W $ is contained in $ C^{u}(x) $ for every $ x \in W $. 

\medskip
\noindent
{\bf Condition H4:} There exist positive constants $ \beta' $ and $ \epsilon_{1} $ such that for every $ u$-manifold $ W $, there exist an integer $ m = m(W) $ and a constant $ B=B(W)>0 $ such that for every $ 0 < \epsilon < \epsilon_{1} $,
\begin{enumerate}
\item $ \nu_{W}(W \cap \mathcal{F}^{-n}(\mathcal{N}^{+}_{\epsilon})) < \epsilon^{\beta'} \nu_{W}(W) \;\; $ for $ n > m $,
\item $ \nu_{W}(W \cap \mathcal{F}^{-n}(\mathcal{N}^{+}_{\epsilon})) < B \epsilon^{\beta'} \nu_{W}(W) \;\; $ for $ n \ge 1 $.
\end{enumerate}
\medskip

For completeness, we provide the proof of the equivalence between Condition~H and Conditions~H3 and~H4.

\begin{lemma}
\label{le:equivalence}
	Condition~H is equivalent to Conditions~H3 and~H4.
\end{lemma}

\begin{proof}
	The fact~H3 and H4 imply H is trivial. We prove the other direction of the equivalence. Since $ a \in (0,1) $ and $ \nu_{W}(W) $ has a uniform upper bound in $ W $, there exists a constant $ B $ such that $ C(a^{n}+\nu_{W}(W)) < B $ for every $ n \ge 1 $ and every $ u $-manifold $ W $. From~H, it follows that if $ \beta'=\beta $ and $ \epsilon_{1} = \epsilon_{0} $, then 
	\[ 
	\nu_{W}(W \cap \mathcal{F}^{-n}(\mathcal{N}_{\epsilon}^{+})) < B \epsilon^{\beta'} 
	\] 
	for every $ n \ge 1 $ and every $ u $-manifold $ W $. Condition~H3 follows from the previous inequality by covering $ \mathcal{U} $ with a smooth family of $ u $-manifolds, and using Fubini's Theorem. Conditions~H4 follows directly from~H by taking 
\[ 
\beta' \in (0,\beta), \qquad \epsilon_{1} = \min \left\{\epsilon_{0},(2C)^\frac{1}{\beta'-\beta}\right\},
\]
and
\[
B(W) = \frac{1+a/\nu_{W}(W)}{2}, \qquad m(W) = \left\lfloor \frac{\log \nu_{W}(W)}{\log a} \right\rfloor,
\]
where $ \lfloor x \rfloor $ is the integer part of $ x $. Indeed, if $ \epsilon<\epsilon_{1} $ and $ n>m(W) $, then 
\begin{align*}
\nu_{W}(W \cap \mathcal{F}^{-n}(\mathcal{N}^{+}_{\epsilon})) & \le C \epsilon^{\beta} \left(a^{n} + \nu_{W}(W)\right) \\ & \le 2C \epsilon^{\beta} \nu_{W}(W) \le \epsilon^{\beta'} \nu_{W}(W).
\end{align*}
Moreover, for every $ n \ge 1 $,
\begin{align*}
\nu_{W}(W \cap \mathcal{F}^{-n}(\mathcal{N}^{+}_{\epsilon})) & \le C \epsilon^{\beta} \left(a^{n} + \nu_{W}(W)\right) \le C \epsilon^{\beta} (a+\nu_{W}(W)) \\ & \le 2C \epsilon^{\beta} B(W) \nu_{W}(W) \le \epsilon^{\beta'} B(W)\nu_{W}(W).
\end{align*} 
\end{proof}

\subsection{SRB measures for polygonal billiards}
Recall that $\Phi_{f,P}$ denotes the billiard map for the polygon $P$ with a contracting reflection law $f$ satisfying the conditions introduced at the end of Section~\ref{sec:billiard map}. As before, we will simply write $ \Phi $ for $ \Phi_{f,P} $ when no confusion can arise. Also, recall that $N^+=V\cup S \cup \partial M$ and $N^-=\partial(\Phi(M\setminus N^+))$. 

The sets $ M,N,N^{\pm},D,D^{\pm}_{\epsilon},A,\ldots $ are the analog for the billiard map $ \Phi $ of the sets $ \mathcal{M},\mathcal{N},\mathcal{N}^{\pm},\mathcal{D},\mathcal{D}^{\pm}_{\epsilon},\mathcal{A},\dots $ for a general piecewise smooth map $ \mathcal{F} $. We also observe that for $ \Phi $, the analog of $ \mathcal{U} $ is the set $ M \setminus \partial M $.

We say that a polygon $P$ has \textit{no parallel sides facing each other} if the endpoints of every straight segment contained inside $P$ and
joining orthogonally two sides of $P$ are vertices of $P$. Notice that $P$ has no parallel sides facing each other if and only if $\Phi$ has no periodic orbits of period two. The reflection law $ f $ does not play any role in the previous claim, because we assumed that $ f(0) = 0 $.

\begin{proposition}\label{prop:hyp} The map $ \Phi_{f,P} $ 
is piecewise smooth satisfying Condition~A1. Moreover, $ \Phi_{f,P} $ satisfies Condition~A2  
if and only if $P$ does not have parallel sides facing each other. 
\end{proposition}

\begin{proof}
The first part is a direct consequence of the fact that the standard billiard map satisfies Condition A1 (see~\cite[Theorem~7.2]{KS86}) and that a reflection law $ f $ together with its inverse has bounded second derivatives. The second claim follows from~\cite[Corollary~3.4]{MDDGP13}.
\end{proof}

\begin{remark}
\label{re:horizontal}
It is easy to see that the horizontal direction ($ \theta = const. $) is always preserved by $ D \Phi_{f,P} $. If $ \Phi_{f,P} $ is uniformly hyperbolic, then the horizontal direction is indeed the expanding direction of $ \Phi_{f,P} $~\cite[Corollary~3.4]{MDDGP13}.	
\end{remark}

We can now state the first part of our main result. 

\begin{theorem}
\label{th:main}
If $P$ does not have parallel sides facing each other, then the conclusions of Theorem~\ref{th:sataev} hold for $ \Phi_{f,P} $ for every contracting reflection law $ f $. Moreover, each Bernoulli component of $ \Phi_{f,P} $ has 
exponential decay of correlations for H\"older observables.
\end{theorem}

The proof of this theorem is given in Subsection~\ref{su:main}.

\subsection{Growth lemma}
We introduce a new condition called $n$-step expansion, and prove that it implies Condition~H. Results of this type are called \emph{growth lemmas} (for instance, see~\cite[Section~5]{CM}). We adopt this terminology. The $ n $-step expansion condition was introduced in~\cite{CZ}. Rather than giving the most general formulation of this condition, we formulate it only for the billiard map $ \Phi $.

\begin{defn}\label{h-curve}
A horizontal open segment contained in $ M \setminus \partial M $ is called a {\em $h$-curve}. 
\end{defn}

\begin{defn}
We say that $\Phi$ satisfies the \textit{$n$-step expansion} condition if there exists $ n \in \Nn $ such that
\begin{equation}
\label{eq:chernov}
\beta(\Phi):= \liminf_{\delta \to 0} \sup_{\Gamma \in \mathcal{H}(\delta)} \sum_{\gamma \in \pi_0(\Gamma\setminus N_n^+)} \frac{1}{a_{n}(\gamma)} < 1,
\end{equation}
where $ \mathcal{H}(\delta) $ is the set of $h$-curves of length less than or equal to $ \delta $, $\pi_0(\Gamma\setminus N_n^+)$ the set of connected components of $\Gamma\setminus N_n^+$ and $ a_{n}(\gamma) $ is the least expansion coefficient of $ D \Phi^{n}|_{(1,0)} $ on $\gamma$, i.e. 
\[ 
a_{n}(\gamma) = \inf_{x \in \gamma} \|D_{x}\Phi^{n}|_{(1,0)}\|.
\]
\end{defn}

Given an $ h $-curve $ \gamma $, we denote by $ \ell_{\gamma} $ the Lebesgue of $ \gamma $. We will drop the index $ \gamma $ in $ \ell_{\gamma} $ when no confusion can arise about which curve $ \gamma $ the measure $ \ell $ refers to.

Recall that $N_\epsilon^+$ denotes the $\epsilon$-neighborhood of $N^+$.

\begin{theorem}[Growth lemma]
\label{thm growth lemma}
If $\Phi$ satisfies the $n$-step expansion condition, then there exist $\beta(\Phi) \le a <1$, $\varepsilon_0>0$ and $C>0$ such that for any $h$-curve $\Gamma$, $r\geq0$ and $0<\varepsilon<\varepsilon_0$, 
\begin{equation}\label{H4}
\ell\left( \Gamma \cap \Phi^{-r}(N_\varepsilon^+)  \right) \leq C \varepsilon (a^{r}+\ell(\Gamma)).
\end{equation}
\end{theorem}

\begin{proof}
Since $ \Phi(M \setminus N) \subset (-\lambda(f) \pi/2,\lambda(f)\pi/2) $ and $ \lambda(f)<1 $, there exists a small $\varepsilon_0$ such that $\varepsilon_0$-neighborhood of $\partial M$ does not intersect $\Phi(M\setminus N^+)$. Therefore, it is enough to prove~\eqref{H4} with $ N_{\varepsilon}^+ $ replaced by $ N_{\epsilon} $.

Choose $\delta>0$ in such a way that
$$
\zeta:=\sup_{\Gamma \in \mathcal{H}(\delta)} \sum_{\gamma \in \pi_0(\Gamma\setminus N_1^+)} \frac{1}{a_{1}(\gamma)} < 1.
$$
Notice that $ \zeta^{1/n} \ge \zeta \geq \beta(\Phi)$. We call an $ h $-curve  \textit{long} if its length is larger or equal than $\delta$, otherwise we call it \textit{short}. 

Let $N_0^+=\emptyset$, and consider an $h$-curve $ \Gamma $. 
By Remark~\ref{re:monotonicity}, the set $\Gamma\cap N_p^+$ consists of $q_p$ elements for every $ p \ge 0 $. Hence,
$ \Phi^{p}(\Gamma\setminus N_p^+)$ is a union of pairwise disjoint $ h $-curves:
$$
\Phi^p(\Gamma\setminus N_p^+)=\bigcup_{i=1}^{q_p} \Gamma_{p,i}.
$$
Clearly $q_p\leq q_{p+1}$, and $q_0=1$ because $\Gamma_{0,1}=\Gamma$.

Write $ r = mn + u \ge 0 $ for $ m \ge 0 $ and $ 0 \le u \le n-1 $. Let $ \Psi = \Phi^{n} $.
Let $I_{k,l}$ be the set of indices $i\in\{1,\dots, q_{mn}\}$ such that 
\begin{enumerate}
\item $\Psi^{-m+k}(\Gamma_{mn,i}) \subset \Gamma_{kn,l}$,
\item $\Psi^{-s}(\Gamma_{mn,i})$ is contained in a short $ h $-curve for $1\leq s\leq m-k-1$. 
\end{enumerate}
Denote by $\LL$ the set of pairs of indices $(k,l)$ with $ k \in \{1,\dots, m\} $ and $ l \in \{1,\dots, q_{kn}\}$ such that $ \Gamma_{kn,l} $ is a long $ h $-curve.
The sets $\{I_{k,l}\}_{(k,l)\in\LL}$ are disjoint, and together with $I_{0,1}$ form a partition of $\{1,\ldots,q_{mn}\}$.
Then, we can estimate $\ell(\Gamma\cap \Phi^{-r}(N_\epsilon))$ as follows:

\begin{align*}
\ell(\Gamma\cap \Phi^{-r}(N_\epsilon))&=\ell(\Gamma\cap \Psi^{-m}(\Phi^{-u}(N_\epsilon)))\\
&=\sum_{i=1}^{q_m}\ell\left(\Psi^{-m}(\Gamma_{mn,i}\cap \Phi^{-u}(N_\varepsilon))\right)\\
&\leq
\ell(\Lambda_{0,1})+\sum_{(k,l)\in\LL} \ell\left(\Psi^{-k}(\Lambda_{k,l})\right),
\end{align*}
where 
$$
\Lambda_{k,l} := \bigcup_{i\in I_{k,l}} \Psi^{-m+k}(\Gamma_{mn,i}\cap \Phi^{-u}(N_\varepsilon)).
$$
Since the restriction of $\Psi^{-m+k}$ to $ \Gamma_{mn,i} $ is affine, we have
$$
\frac{\ell\left(\Psi^{-k}(\Lambda_{k,l})\right)}{\ell\left(\Psi^{-k}(\Gamma_{kn,l})\right)} = 
\frac{\ell(\Lambda_{k,l})}{\ell( \Gamma_{kn,l})}.
$$
Since $ \Gamma_{kn,l} $ is long for $ (k,l) \in \mathcal{L} $, it follows that for any $ r \ge 0 $ and any $0<\varepsilon<\varepsilon_0$,
$$
\ell\left(\Gamma\cap \Phi^{-r}(N_\varepsilon) \right) \leq \ell(\Lambda_{0,1})+
\frac1\delta \sum_{(k,l)\in\LL} \ell(\Lambda_{k,l})  \,\ell\left(\Psi^{-k}(\Gamma_{kn,l})\right).
$$

We now estimate $\ell(\Lambda_{k,l})$. Let $ t=m-k $. Given $ i \in I_{k,l} $, let $ i_{s} \in \{1,\ldots,q_{(m-t+s)n} \} $ be the index defined by $ \Psi^{-t+s}(\Gamma_{mn,i}) \subset \Gamma_{(m-t+s)n,i_{s}} $ for $ 1 \le s \le t $. Also, let $ a(s,i_{s}) $ be the least expansion of $ \Psi $ along the curve $ \Psi^{-1}(\Gamma_{(m-t+s)n,i_{s}}) $ for $ 1 \le s \le t $. Thus,
\begin{align*}
\ell(\Lambda_{k,l}) & \le \sum_{i \in I_{k,l}} \ell \left(\Psi^{-t}(\Gamma_{mn,i}\cap \Phi^{-u}(N_\varepsilon)) \right) \\
& \le \sum_{i \in I_{k,l}} \frac{\ell(\Gamma_{mn,i}\cap \Phi^{-u}(N_\varepsilon))}{a(1,i_{1}) \cdots a(t,i_{t})}.
\end{align*}

The uniform transversality between $N$ and the horizontal direction implies that there exists a constant $C'>0$ independent of $\Gamma$, $ r $ and $ i' $ such that 
\[
\ell(\Gamma_{r,i'}\cap N_\varepsilon) \leq C'\varepsilon.
\]

Let $ d>0 $ be the maximum number of intersection points of $ N $ with $ h $-curves. Also, let $ b>0 $ be the least expansion of $ \Phi $ along the horizontal direction. Note that $ b $ is not necessarily greater than 1, and that $ d $ and $ b $ depend only on $ \Phi $. Then, each $ \Gamma_{mn,i} $ contains at most $ (d+1)^{u} $ curves $ \Phi^{-u}(\Gamma_{r,i'}) $. This together the previous estimate implies 
\[
\ell(\Gamma_{mn,i}\cap \Phi^{-u}(N_\varepsilon)) \le \frac{(d+1)^{u}}{b^{u}} C'\epsilon.
\]

Define
$ I_{1} = \{i_{1} \colon i \in I_{k,l}\} $,
and 
$ I_{s} =\{(i_{s-1},i_{s}) \colon i \in I_{k,l} \} $ for $ 1 \le s \le t $.
Then, we have
\begin{align*}
\sum_{i \in I_{k,l}} & \frac{1}{a(1,i_{1}) \cdots a(t,i_{t})} \\ & =	\sum_{j_{1} \in I_{1}} \sum_{(j_{1},j_{2}) \in I_{2}} \cdots \sum_{(j_{t-1},j_{t}) \in I_{t}} \frac{1}{a(1,j_{1}) \cdots a(t,j_{t})} \\
& = \sum_{j_{1} \in I_{1}} \frac{1}{a(1,j_{1})}
 \sum_{(j_{1},j_{2}) \in I_{2}} \frac{1}{a(2,j_{2})}
 \cdots \sum_{(j_{t-1},j_{t}) \in I_{t}} \frac{1}{a(t,j_{t})}
\end{align*}
Since the curve $ \Gamma_{(m-t+s-1)n,j_{s-1}} \supseteq \Psi^{-1}(\Gamma_{(m-t+s)n,j_{s}}) $ is short, \eqref{eq:chernov} implies
\[
\sum_{(j_{t-1},j_{t}) \in I_{t}} \frac{1}{a(s,j_{s})} \le \zeta, \qquad 2 \le s \le t.
\]
The same argument does not necessarily apply to $ \sum_{j_{1} \in I_{1}} 1/a(1,j_{1}) $, because the curve $ \Gamma_{kn,l} \supseteq \Psi^{-1}(\Gamma_{(m-t+1)n,j_{1}}) $ may be long. However, we have
\[
\sum_{j_{1} \in I_{1}} \frac{1}{a(1,j_{1})} \le \frac{d+1}{b}.
\]
Hence,
\[
\sum_{i \in I_{k,l}} \frac{1}{a(1,i_{1}) \cdots a(t,i_{t})} \le \frac{d+1}{b} \zeta^{t-1},
\]
and so
\[
\ell(\Lambda_{k,l}) \le C'' \epsilon \, \zeta^{m-k}, \qquad C'' = \frac{C'}{\zeta} \left(\frac{d+1}{b}\right)^{u+1}.
\]

The above estimate implies
$$
\ell\left(\Gamma\cap \Phi^{-r}(N_\varepsilon) \right)
\leq
C''\varepsilon\,\zeta^m + \frac{C''\varepsilon}{\delta} \sum_{k=1}^m \zeta^{m-k} 
\sum_{l=1}^{q_{kn}} \ell\left(\Psi^{-k}(\Gamma_{kn,l})\right).
$$
Since 
\[
\sum_l \ell\left(\Psi^{-k}(\Gamma_{kn,l})\right)= \ell\left(\Psi^{-k}\left(\bigcup^{q_{kn}}_{l=1}\Gamma_{kn,l}\right)\right)\leq\ell(\Gamma),
\]
we obtain
\begin{align*}
\ell\left(\Gamma \cap \Phi^{-r}(N_\varepsilon)\right) & \leq
C''\varepsilon\,\zeta^m + \frac{C''\,\varepsilon\,\ell(\Gamma)}{\delta(1-\zeta)} \\
& \le C''\varepsilon \, \zeta^{\frac{r}{n}-1} + \frac{C''\,\varepsilon\,\ell(\Gamma)}{\delta(1-\zeta)},
\end{align*}
which implies the wanted conclusion.
\end{proof}

\subsection{Proof of Theorem~\ref{th:main}}
\label{su:main}
In this subsection, we prove the $ n $-step expansion for the billiard map and Theorem~\ref{th:main}. 

Since we assume that the polygon $ P $ does not have parallel sides facing each other, the map $ \Phi $ is uniformly hyperbolic by Proposition~\ref{prop:hyp}. 

Define the least expansion rate of $D\Phi^{n}$ along the unstable direction by
\[
A_{n}:=\inf_{x \in M^+} \|D_{x}\Phi^{n}|_{(1,0)}\|.
\]

Also define 
\begin{equation}
\label{eq:bound}	
 \alpha(x) := \frac{\cos \theta(x)}{\cos \bar{\theta}_{1}(x)} \ge \cos\left(\frac{\pi}{2}\lambda(f)\right),
\end{equation}
and 
\begin{equation}
\label{eq:expansion}	
\alpha_{n}(x) := \|D_{x} \Phi^{n}|_{(1,0)}\| =\alpha(x) \cdots \alpha(\Phi^{n-1}(x)).
\end{equation}

\begin{lemma}
\label{le:divergent}
Suppose that $ \Delta $ is a primary non-regular sector of order $ n $ with vertex $ x \notin \partial M $. If $ \Delta' $ is a sub-sector of $ \Delta $ of order $ m > n $, or $ \Delta' = \Delta $ and $ m=n $, then
\[
\lim_{\Delta' \ni y \to x} \alpha_{m}(y) = +\infty.
\]
\end{lemma}

\begin{proof}
By Remark~\ref{re:non-regular}, the map $ \Phi^{n-1} $ is a $ C^{2} $ diffeomorphism on the closure of $ \Delta $, and $ \Phi^{n-1}(x) $ is a tangential singularity. Thus, $ \bar{\theta}_{1}( \Phi^{k}(x)) = \pm \pi/2 $, and so $ \alpha(\Phi^{n-1}(y)) \to +\infty $ as $ \Delta \ni y \to x $. The claim now follows from~\eqref{eq:bound},~\eqref{eq:expansion} and $ \Delta' \subseteq \Delta $. 
\end{proof}

\begin{proposition}
\label{pr:chernov}
The map $ \Phi $ has $n$-step expansion for every $ n $ sufficiently large.
\end{proposition}

\begin{proof}
Given a sector $ \Delta $ of order $ \le n $ with vertex $ x $, let 
\[ 
\Delta_{\epsilon} = \{y \in \Delta(x) \colon \dist(y,x)<\epsilon \} 
\] 
for $ \epsilon>0 $. Denote by $ C_{\epsilon} $ the union of all $ \Delta_{\epsilon} $ with $ \Delta $ being a primary non-regular sector of order $ \le n $ with vertex not belonging to $ \partial M $.

For a fixed $ \Gamma \in \mathcal{H}(\delta) $, we have 
\[
\sum_{\gamma\in \pi_0(\Gamma\setminus N_n^+)} \frac{1}{a_{n}(\gamma)} = \sum_{\gamma'} \frac{1}{a_{n}(\gamma')} + \sum_{\gamma''} \frac{1}{a_{n}(\gamma'')},
\]
where $ \sum_{\gamma'} $ and $ \sum_{\gamma''} $ denote the sum over the components of $ \Gamma \setminus N^{+}_{n} $ intersecting the complement of $ C_{\epsilon} $ and contained in $ C_{\epsilon} $, respectively. From the definition of $ b_{n} $, it follows that there are at most $b_n$  connected components of $ \Gamma \setminus N^{+}_{n} $ contained in the complement of $C_\epsilon$ provided that $\delta$ is sufficiently small.

Since $ \Phi $ is hyperbolic, there exist $ c>0 $ and $ \Lambda>1 $ such that $ A_{n} \ge c \Lambda^{n} $ for every $ n \in \Nn $, and $ b_{n} $ grows linearly in $ n $ by Theorem~\ref{th:linear}, we can find $ n_{0} \in \Nn $ such that $ A_{n} > b_{n} $ for every $ n \ge n_{0} $. Choose $ n \ge n_{0} $, which will be kept fixed throughout the rest of the proof. 

Now, $ N^{+}_{n} $ consists of finitely many curves that either are disjoint or intersect pairwise at finitely many points (see the proof of~\cite[Proposition~2.3]{MDDGP15}). Thus, there exists $ d(\epsilon) \to 0 $ as $ \epsilon \to 0 $ such that any two distinct components of $ C^{c}_{\epsilon} \cap N^{+}_{n} $ either have a distance greater than $ d(\epsilon) $ or meet at a point that is not a tangential singularity. In view of the choice of $ n $ and the fact that $ a_{n}(\gamma) \ge A_{n} $, there exists $ 0<\eta<1 $ such that
\begin{equation}
\label{eq:one}	
\sum_{\gamma'} \frac{1}{a_{n}(\gamma')} \le \frac{b_{n}}{A_{n}} = 1-\eta,
\end{equation}
provided that $ \delta < d(\epsilon) $.

From Lemma~\ref{le:divergent}, we have $ a_{n}(\gamma'') \to +\infty $ as $ \epsilon \to 0 $ for every $ \gamma'' $. Thus, by choosing $ \epsilon $ sufficiently small, we can make sure that 
\begin{equation}
\label{eq:two}	
\sum_{\gamma''} \frac{1}{a_{n}(\gamma'')} < \frac{\eta}{2}.
\end{equation}

Combining~\eqref{eq:one} and \eqref{eq:two}, which do not depend on $ \Gamma $, we obtain 
\[
\sum_{\gamma} \frac{1}{a_{n}(\gamma)} < 1-\frac{\eta}{2}
\]
for every $ \delta $ sufficiently small and for every $ \Gamma \in \mathcal{H}(\delta) $. This implies \eqref{eq:chernov}, and completes the proof.
\end{proof}

\begin{proof}[Proof of Theorem~\ref{th:main}]
First of all, we observe that for $ \mathcal{F} = \Phi $, Theorem~\ref{th:sataev} remains valid if Condition~H is satisfied only by $ h $-curves, since the local unstable manifolds of $ \Phi $ are $ h $-curves (see \cite[Proposition 2.7]{MDDGP15}). 
By Proposition~\ref{prop:hyp}, the map $ \Phi $ satisfies Conditions~A1 and A2. By Proposition~\ref{pr:chernov} and Theorem~\ref{thm growth lemma}, Condition~H holds for $ \Phi $ and $h$-curves. Hence, Theorem~\ref{th:sataev} applies to $\Phi$.

We now prove the second claim of the theorem. Consider a Bernoulli component $ (\Phi^{k_{i}}|_{\mathcal{M}_{i,j}},\mu_{i,j}) $ of $ \Phi $, and let $ \Psi = \Phi^{k_{i}}|_{\mathcal{M}_{i,j}} $. By Proposition~\ref{pr:chernov}, the $ n $-step expansion condition holds true for $ \Psi $ for some $ n \ge k_{i} $. It follows from~\cite[Proposition~10.1]{C} (see also~\cite[Theorem~10]{CZ}) that it is enough to establish the exponential decay of correlations for $ (\Psi^{n},\mu_{i,j}) $. To do that, we apply a theorem of Chernov and Zhang~\cite[Theorem~1]{CZ2009} to $ \Psi^{n} $. This theorem has five hypotheses~H.1-H.5. It is not difficult to see that~H.1 and~H2 follow from our Conditions~A2 and~A1, respectively. The finiteness of the number of smooth components of $ N^{+} $ and $ N^{-} $ follows from Remark~\ref{re:monotonicity} and the first part of the proof of Lemma~\ref{le:increasing}. Hypothesis H.3 is satisfied if we take as the $ \Psi^{n} $-invariant class of smooth $ u $-curves the set of all $ h $-curves. Indeed, this class satisfies the three conditions of H.3: i) the curvature of the $ h $-curves is clearly uniformly bounded, ii) the restriction of $ D \Psi^{n} $ along $ h $-curves has uniform distortion bounds because the restriction of $ \Psi^{n} $ to an $ h $-curve is a piecewise affine map, and iii) by Lemma~\ref{Lipschitz holomomy} (see also~\cite[Proposition 10]{pe}), the stable holonomy is absolutely continuous. Since $ (\Psi,\mu_{i,j}) $ is Bernoulli, hypothesis~H.4 is trivially satisfied. Thus, using~\cite[Theorem~1]{CZ2009}, we conclude that $ (\Psi^{n},\mu_{i,j}) $ has exponential decay of correlations for H\"older observables. The same is true for $ (\Psi,\mu_{i,j}) $ by~\cite[Proposition~10.1]{C}.
\end{proof}

%% file: basins.tex
Recall that $\Phi=\Phi_{f,P}$ denotes the billiard map for the polygon $P$ with contracting reflection law $f$. The aim of this section is to prove the following theorem.

\begin{theorem}
\label{thm  union of basins}
Let $ P $ be a polygon without parallel sides facing each other. Then the union of the basins of the ergodic SRB measures of $\Phi$ has full Lebesgue measure.
\end{theorem}

Choose $\epsilon>0$ so that Proposition~\ref{prop:stable manifolds}  and Theorem~\ref{thm growth lemma} hold. From now on, to simplify our notation, we will drop the index $\epsilon$ from the symbols $\hat D^+_{\epsilon,l}$, $D^-_{\epsilon,l}$, $D^0_{\epsilon,l}$. 

To prove Theorem~\ref{thm union of basins}, we will follow the proof of~\cite[Proposition 4.2]{bovi}, where a similar result is proved for smooth maps.

\begin{defn}
Given a $C^1$ simple open curve $\gamma$ in $M$ and a point $x\in\gamma$, we say that \textit{$\gamma$ has size $\delta$ around $x$} if the length of the connected components of $\gamma\setminus\{x\}$ is greater than or equal to $\delta$. 
\end{defn}

\begin{defn}
Given an $h$-curve $\gamma$, a point $x\in \gamma $ and $ n \in \Nn $, let $R^n_\gamma (x)$ be the connected component of $\Phi^n (\gamma)$ containing $\Phi^n (x)$. Define
$$\gamma(n,\delta) =\left\{y\in\gamma \colon
 R_\gamma^n(y) \text{ has size } \delta \text{ around } \Phi^n(y)\right\}.$$
\end{defn}

Let $ \epsilon_{0}>0 $ and $ 0 < a < 1 $ be as in Condition~H, and let $ c>0 $ and $ \lambda>1 $ be as in Condition~A2.

\begin{lemma}
\label{lemma size delta cc}
There exists $ \widetilde{C} = \widetilde{C}(\Phi) > 0 $ such that for every $ h $-curve $ \gamma $, there exists $ n_{0} = n_{0}(\gamma) $ such that if $ n \ge n_{0} $ and $ 0 < \delta < c \epsilon_{0} $, then 
\[
\ell \left(\gamma \setminus \gamma(n,\delta) \right) \le \widetilde{C} \ell (\gamma) \delta. 
\]
\end{lemma}

\begin{proof}
Let $ \gamma $ be an $ h $-curve. If $ x \in \gamma \setminus \gamma(n,\delta) $, then either $ x \in N^{+}_{n} $ or $ d(\Phi^{n}(x),\partial \Phi^{n}(\gamma \setminus N^{+}_{n}))<\delta $. Let $ B \subset \partial \Phi^{n}(\gamma \setminus N^{+}_{n}) $ be the image under $ \Phi^{n} $ of the endpoints of $ \gamma $ that do not belong to $ N^{+}_{n} $. Note that the elements of $ \partial \Phi^{n}(\gamma \setminus N^{+}_{n}) \setminus B $ are discontinuities of $ \Phi^{-n} $. Since the horizontal direction coincides with the unstable direction and $ \|D \Phi^{n}|_{E^{u}}\| \ge c \lambda^{n} $,
\[
\gamma \setminus \gamma(n,\delta) \subset \bigcup^{n-1}_{i=0} \left(\gamma \cap \Phi^{-i}\left(N^{+}_{\delta \lambda^{-n+i}/c} \right) \right) \cup \left(\gamma \cap \Phi^{-n}\left(B_{\delta}\right)\right).
\]
Using~\eqref{H4}, we obtain
\begin{align*}
\ell \left(\gamma \setminus \gamma(n,\delta) \right) & \le \sum^{n-1}_{i=0} \ell \left(\gamma \cap \Phi^{-i}\left(N^{+}_{\delta \lambda^{-n+i}/c} \right) \right) + \ell\left(\gamma \cap \Phi^{-n}\left(B_{\delta}\right)\right) \\ & \le \sum^{n-1}_{i=0} \frac{C \delta}{c \lambda^{n-i}} \left(a^{i} + \ell(\gamma) \right) + \frac{2\delta}{c \lambda^{n}} \\ & \le \frac{c \delta}{c \lambda^{n}} \sum^{n-1}_{i=0} \left( (a \lambda)^{i} + \lambda^{i} \ell(\gamma) \right) + \frac{2 \delta}{c \lambda^{n}} \\ & \le \frac{C \delta}{c \lambda^{n}} \left( \frac{1-(a \lambda)^{n}}{1 - a \lambda} + \frac{1 - \lambda^{n}}{1 - \lambda} \ell(\gamma) \right) + \frac{2 \delta}{c \lambda^{n}}   \\ & \le \frac{C \delta}{c} \left( \frac{\lambda^{-n}-a^{n}}{1 - a \lambda} + \frac{\lambda^{-n} - 1}{1 - \lambda} \ell(\gamma) + \frac{2}{C \lambda^{n}}  \right)  \\ & \le \frac{2C}{c} \ell(\gamma) \delta
\end{align*}
for all $ n $ such that the expression in parentheses in the penultimate inequality is smaller than $ 2 \ell(\gamma) $.
\end{proof}

Let $\hat D^+=\bigcup_{l\in\Nn}\hat D_l^+$.
 
\begin{lemma}
\label{lemma gamma D^+}
For any $h$-curve $\gamma$, we have $\ell (\gamma)=\ell (\gamma\cap \hat D^+)$.
\end{lemma}

\begin{proof}
Since $ D^{+}_{l} $ is increasing in $ l $, we have $ \ell(\gamma \setminus \hat{D}^{+}_{l}) \to \ell(\gamma \setminus \hat{D}^{+}) $ as $ l \to +\infty $. Now,
\[
\gamma \setminus \hat{D}^{+}_{l} \subset \bigcup^{\infty}_{n=0} \left(\gamma \cap \Phi^{-n} \left(N^{+}_{l^{-1}e^{-\epsilon n}} \right) \right)
\]
If $ l $ is sufficiently large, then by~\eqref{H4},
\begin{align*}
	\ell(\gamma \setminus \hat{D}^{+}_{l}) & \le \sum^{\infty}_{n=0} \ell \left(\gamma \cap \Phi^{-n} \left(N^{+}_{l^{-1}e^{-\epsilon}} \right) \right) \\
	& \le \frac{C}{l} \sum^{\infty}_{n=0} e^{-\epsilon n} \left(a^{n} + \ell(\gamma)\right) \le  \frac{C}{l} \cdot \frac{1+\ell(\gamma)}{1-e^{-\epsilon}} \xrightarrow[l \to +\infty]{} 0.
\end{align*}
\end{proof}

Given an $h$-curve $\gamma$, we call {\it $\gamma$-limit measure} any weak-$\ast$ accumulation point of the averages
$$ \mu_{\gamma,n}:= \frac{1}{n}\,\sum_{j=0}^{n-1}\Phi_\ast^j\,\Leb_\gamma,$$
where $\Leb_\gamma$ denotes the normalized Lebesgue measure of $\gamma$, i.e. given any Borel set $A\subset M$,
$$
\Leb_\gamma(A)=\frac{\ell(\gamma\cap A)}{\ell(\gamma)}.
$$

The next is the key result on the existence of SRB measures proved in~\cite{pe}. We recall part of the proof, for the convenience of the reader.

\begin{lemma}
\label{lemma mu D^0}
If $ \gamma $ is an $h$-curve, then any $\gamma$-limit measure $\mu$ is invariant and satisfies $\mu(D^0)=1$. In addition, if $\gamma\cap D^-\not=\emptyset$, then $\mu$ is an SRB measure.
\end{lemma}

\begin{proof}
Consider the function
$\varphi:M\setminus N^{+}\to\Rr$,
$$\varphi(x):=  \dist(x,N^{+}) \;,$$
where the distance between $x$ and $N^{+}$ is measured along the horizontal line through $x$. Since $\varphi$ is bounded from above,
for any probability measure $\nu$, the integral $\int \log\varphi\,d\nu$ is well-defined in $[-\infty,+\infty)$.
Notice also that for all sub-intervals $I$ of some fixed compact interval containing the origin,
$\int_I \log\abs{x}\,dx \geq -\len{I}$. Hence, by a change of coordinates
$$  \int_{\gamma\cap \Phi^{-j}(N^{+}_\epsilon)}
\log \dist(\Phi^j(x),N^{+})\, dx  \geq
-\len{\gamma\cap \Phi^{-j}(N^{+}_\epsilon)}\;, $$
and by ~\eqref{H4} there exists a constant $C>0$ such that  

\begin{align*} 
\int \log \varphi\,d\mu_{\gamma,n} 
&= \int_{N^{+}_\epsilon} \log \varphi\,d\mu_{\gamma,n}  +
\int_{M\setminus N^{+}_\epsilon} \log \varphi\,d\mu_{\gamma,n}  \\
 &\geq  - \frac{1}{n}\,\sum_{j=0}^{n-1}
\frac{C\,\epsilon\,(a^j +\len{\gamma})}{\len{\gamma}} 
 - \log(1/\epsilon) \\
  &\geq  - C\,\epsilon - \frac{C\,\epsilon}{n \,\len{\gamma}\,(1-\eta)} 
 - \log(1/\epsilon) >-\infty \;.
\end{align*}
Because this lower bound is uniform in $n$,
the function $\log \varphi$ is $\mu$-integrable, with 
$\int \log\varphi\,d\mu \geq - C\,\epsilon - \log(1/\epsilon) $.
Similarly, $\log \varphi\circ \Phi^k$ is $\mu$-integrable for every $k\in\Zz$.
Hence, $\mu(\Phi(N^{+}))=0$.

To prove that $\mu$ is invariant it is enough to see, for any continuous function $\psi:M\to\Rr$, that
\begin{equation}
\label{mu invariance rel}
\int \psi \, d\mu = \int \psi\circ \Phi\, d\mu  \;. 
\end{equation}  
If  $\psi=0$ on $\Phi(N^{+})$  then the composition $\psi\circ\Phi$ is also continuous on $M$, and hence 
$$ \lim_{k\to+\infty} \int \psi\circ \Phi\, d\mu_{\gamma,n_k} = \int \psi\circ \Phi\, d\mu \;. $$
On the other hand,  a standard calculation gives
\begin{align*}
\abs{ \int \psi \, d\mu_{\gamma,n} - \int \psi\circ \Phi\, d\mu_{\gamma,n} }\leq 
 \frac{2}{n}\,\norm{\psi}_\infty.
\end{align*}
Passing to the limit we get~\eqref{mu invariance rel}.

In general, given any continuous function $\psi:M\to\Rr$, because $\mu(\Phi(N^{+}))=0$, we can approximate $\psi$ by continuous functions
$\psi'$ vanishing on $\Phi(N^{+})$, and taking the limit we obtain the
invariance relation~\eqref{mu invariance rel} for $\psi$.
This proves that $\mu$ is invariant. 

By Lemma~\ref{lemma gamma D^+},
$\Leb_\gamma(\hat D^+)=1$. Hence because $\Phi(\hat D^+)\subseteq \hat D^+$, we have  $\mu_{\gamma,n}(\hat D^+)=1$
for all $n\geq 1$. Thus, taking the limit, $\mu(\hat D^+)=1$.

We are now left to prove that $\mu(D^-)=1$. From Birkhoff's ergodic theorem,  for $\mu$-almost every $x\in M$,
$ \lim_{n\to +\infty} \frac{1}{n}\,\log \varphi(\Phi^{-n} x)=0$.
This implies that there exists $l\in\Nn$ such that $\varphi(\Phi^{-n} x)\geq l^{-1}\,e^{-n\,\epsilon}$ for all $n\in\Nn$. Hence,
$$ \dist(\Phi^{-(n-1)} x, N^-)\geq 
\dist(\Phi^{-n} x, N^{+}) = \varphi(\Phi^{-n} x)\geq l^{-1}\,e^{-n\,\epsilon}\;, $$
which proves that $x\in D^-_l$ for some $l\in\Nn$. Thus
 $\mu(D^-)= 1$.  The last claim follows  from the proof of existence of SRB measures in~\cite[Theorem 1]{pe}.
\end{proof}

In the next lemma, we prove that the stable holonomy is Lipschitz continuous (c.f.~\cite[Proposition 10]{pe}).

\begin{lemma} \label{Lipschitz holomomy}
Given $l\in\Nn$, there are constants $\tilde\delta>0$ and $C>0$ such that
if $\Gamma$ and $\Gamma'$ are $ h $-curves whose distance is less than $\tilde{\delta}$, $x_1,x_2\in\Gamma\cap \hat D^+_{l} $ with 
$\abs{x_1-x_2}<\tilde{\delta}$ and $x_i':= \Gamma'\cap W^s_{l}(x_i)$, then
$$ \frac{\abs{x_1'-x_2'}}{\abs{x_1-x_2}}\leq C.$$
\end{lemma}

\begin{proof}
The proof consists of a few steps.

\smallskip

\noindent
(a)\;  First the slope of local stable manifolds is uniformly bounded away from $0$. This follows easily from the expression for tangent space to the stable manifold in the proof of ~\cite[Proposition 3.1]{MDDGP13}.

\smallskip

\noindent
(b)\;  Let $\Gamma_n(x_1)$ denote the $\Phi^n$
pre-image of the connected component of
$\Phi^n(\Gamma)$ that contains $\Phi^n(x_1)$. 
Let $\Gamma_n'(x_1)$ be the corresponding component of $\Gamma'$ w.r.t. $x_1'$.
Denoting  $\alpha_n = (\Phi^n\vert_{\Gamma_n(x_1)})'$, resp. $\alpha_n'=(\Phi^n\vert_{\Gamma_n'(x_1')})'$, then
$$ \abs{ \log  \frac{\alpha_n}{\alpha_n'}  } 
\leq C\,\delta\;.$$

Notice that $\alpha_n(s_0,\theta_0)
=\prod_{j=0}^{n-1} \rho(\theta_i)$, where
$\rho(\theta):=cos f(\theta)/\cos \theta$
and for each $i\geq 1$, $(s_i,\theta_i)$ is the image of
$(s_{i-1}, f(\theta_{i-1}))$ by the specular billiard.

Given another point $(s_0',\theta_0')\in W^s_{{\rm loc}}(s_0,\theta_0)$, denote by $\theta_i'$ the analogous angles for $(s_0',\theta_0')$. Since $\log \rho(\theta)$ is Lispchitz, and $\abs{\theta'_j-\theta_j}$ decay geometrically with $j$, we obtain
\begin{align*}
\abs{\log \frac{\alpha_n(s_0, \theta_0)}{\alpha_n(s_0', \theta_0')}} &\leq\sum_{j=0}^{n-1}
\abs{\log \rho(\theta_i)-\log \rho(\theta_i')}\\
& \leq  \sum_{j=0}^{n-1} C\,\abs{\theta_i-\theta_i}
\lesssim C\,\abs{\theta_0-\theta_0'}
=C\,\delta\;. 
\end{align*}

\smallskip

\noindent
(c)\;  Write $x_{1,n}=\Phi^n(x_1)$,
$x_{2,n}=\Phi^n(x_2)$, and consider the first $n\geq 1$ such that $x_{1,n}$ and $x_{2,n}$ do not belong to the same branch domain of $\Phi$.
Because $x_i\in \hat D^+_{l}$, and these two points are separated by singular curve, we have
$\abs{x_{1,n}-x_{2,n}}\leq 2\,l^{-1}\,e^{-n\,\epsilon}$.
Writing $x_{1,n}'=\Phi^n(x_1')$ and $x_{2,n}'=\Phi^n(x_2')$, we have
$\abs{ x_{1,n} - x_{1,n}' } = 1/\alpha_n(x_1)\ll l^{-1}\,e^{-n\,\epsilon}$. Hence, combining this information with (a), we get  
$$ \frac{\abs{x_{1,n}'-x_{2,n}'}}{\abs{x_{1,n}-x_{2,n}}} \leq C\;. $$

\smallskip

\noindent
(d)\; By parts (b) and (c) and the mean value theorem, we obtain
\begin{align*}
\frac{\abs{x_{1}'-x_{2}'}}{\abs{x_{1}-x_{2}}} &\asymp \abs{\frac{\alpha_n}{\alpha_n'}}\,\frac{\abs{x_{1,n}'-x_{2,n}'}}{\abs{x_{1,n}-x_{2,n}}} \\
&\leq  e^{C\,\delta} \frac{\abs{x_{1,n}'-x_{2,n}'}}{\abs{x_{1,n}-x_{2,n}}} \leq C\,e^{C\,\delta}\;. 
\end{align*}

\end{proof}

By Proposition~\ref{prop:stable manifolds}, for every $l\in\Nn$, there exists $\delta_l>0$ such that any point $x\in \hat D^+_l$ (resp. $x\in D^-_l$)  has a local stable (resp. unstable) curve of size $\delta_l$ around $x$, denoted by $W^s_{l}(x)$ (resp. $W^u_{l}(x)$). For any set $A\subset D^0_l$, define $W^s_l(A) =\bigcup_{x\in A} W^s_{l}(x)$.

\begin{lemma}
\label{lemma gamma cap B mu >0 }
Let $\gamma$ be an $h$-curve such that $\gamma\cap D^-\not=\emptyset$. If $\mu$ is a $\gamma$-limit measure, then there exists an ergodic component $\mu_c$ of $\mu$ such that $\Leb_\gamma\left(B(\mu_c)\right)>0$.
\end{lemma}

\begin{proof}
By Lemma~\ref{lemma mu D^0}, $\mu$ is an SRB measure of $ \Phi $.  According to Theorem~\ref{th:main}, we may decompose $\mu$ into a finite number of ergodic components $\mu_i$, $i=1,\ldots, r$ such that $\mu$ is a convex combination of the ergodic measures $\mu_i$, i.e. $\mu=\sum_i\alpha_i\mu_i$. If $B:=\bigcup_{i=1}^r B(\mu_i)$, then
\begin{equation}\label{eq:mu B}
\mu\left(B\right)=1.
\end{equation}

Given $l\in\Nn$ define
$$  \Omega_{l}^0 := \left\{  x\in D^0_l \colon \Leb_{W_l^u(x)}\left( D^0_l \cap B\right) >1-\frac{\delta}{10}  \right\} 
$$
where 
$ \delta = \delta(l) $ is smaller than $ \tilde{\delta} $ as in Lemma~\ref{Lipschitz holomomy} and to be determined later.
Notice that $\mu(D^0\cap B)=1$ due to Lemma~\ref{lemma mu D^0} and \eqref{eq:mu B}. 
By the properties of an SRB measure (see~\cite[Proposition 9]{pe})
, we obtain 
$$
\mu\left(\bigcup_{l\in\Nn} \Omega_{l}^0\right)=1.
$$

Let $B(x,\delta)$ be the open ball of size $\delta$ around $x\in M$. For every $x\in \Omega_{l}^0$, denote by $\Pi_{l}(x)$ the union of all $h$-curves  of size $\delta/10$ centered at $y\in W_l^s(x)\cap B(x,\delta)$. 

The set 
$$
U_{l}:=\bigcup_{x\in \Omega^0_{l}}\Pi_{l}(x)
$$
is open and contains $\Omega_{l}^0$.
Choose $l\in\Nn$ sufficiently large such that
 $$
 \mu\left(\Omega_{l}^0\right)>\frac{9}{10}.
 $$

Since $\mu$ is a $\gamma$-limit and $U_{l}$ is open, we have
$$
\liminf_n \mu_{\gamma,n}(U_{l})\geq \mu(U_{l})\geq \mu(\Omega_{l}^0)>\frac{9}{10}.
$$
Hence, for infinitely many $n\in\Nn$,
\begin{equation}\label{eq:Lebgamma}
 \Leb_\gamma( \Phi^{-n}(U_{l}))) 
 > \frac{9}{10}.
\end{equation}
Notice that $\delta(l)\to 0$ as $l\to\infty$. According to Lemma~\ref{lemma size delta cc}, for any $l\in\Nn$ sufficiently large there exists $n_0\in\Nn$ such that for every $n\geq n_0$,
$$
\Leb_\gamma\left(\gamma(n,\delta)\right) > \frac{9}{10}.
$$
This lower bound together with \eqref{eq:Lebgamma} implies that there exist $l\in\Nn$ and $n\in\Nn$ such that
$$
\Leb_\gamma\left(\gamma(n,\delta)\cap \Phi^{-n}(U_{l})\right) > \frac{4}{5}.
$$
Therefore,  $\Phi^n(y)\in \Pi_{l}(x)$ for some $x\in \Omega_{l}^0$ and $y\in\gamma(n,\delta)$.

Recall that $R_\gamma^n(y)$ is the $h$-curve of size $\delta$ around $\Phi^n(y)$. Hence, we can find $z\in\gamma$ such that $\Phi^n(z)\in W_l^s(x)$ and an $h$-curve contained in $R_\gamma^n(y)$ of size $\frac{9}{10}\delta$ around $\Phi^n(z)$. Moreover, since $x\in \Omega_{l}^0$ we know that,
$$
\Leb_{W_l^u(x)}\left( D^0_l \cap B \right) >1-\frac{\delta}{10}.
$$
This means that $W_l^s( W_l^u(x)\cap D^0_l \cap B)$ has no `vertical gaps' having size larger than $\delta/10$. 
The local stable manifolds of points in $ D^{0}_{l} $ vary continuously. Thus, by choosing a sufficiently small $ \delta(l) $ and using the absolute continuity of the stable holonomy, we can make sure
that $R_\gamma^n(y)$ intersects $W_l^s(W_l^u(x)\cap D^0_l \cap B)$ on a set of positive $\ell$-measure. 
Notice that if $ z \in B \cap D^{0}_{l} $, then $ W^{s}_{l}(z) \subset B $. Hence, $\Leb_{R_\gamma^n(y)} \left(B\right)>0 $. This implies $\Leb_\gamma\left(B(\mu_i)\right)>0$ for some $i=1,\ldots,r$, because the restriction of $\Phi^{-n}$ to $R^n_\gamma(y)$ is affine.

\end{proof}

\begin{lemma}\label{len gamma cap B(mu) >0}
If $\gamma$ is an $h$-curve, then there exists an ergodic SRB measure $\mu_c$ such that $\Leb_\gamma\left(B(\mu_c)\right)>0$.
\end{lemma}

\begin{proof}
In order to apply Lemma~\ref{lemma gamma cap B mu >0 } we produce a sub-limit $\Gamma$ of iterates of $\gamma$ such that $\Gamma\cap D^-\not=\emptyset$.

Let $\mu$ be a $\gamma$-limit measure. By Lemma~\ref{lemma mu D^0}, we can take $l\in\Nn$ such that $\mu(D^0_l)>\frac{9}{10}$. By Lemma~\ref{lemma size delta cc}, there exists $\delta>0$, which we assume to be much smaller than the size of the local stable curves of points in $D^0_l$, such that for all large $n$,
$$\Leb_\gamma\left(\gamma(n,\delta)\right)>\frac{9}{10}.$$

Denote by $\lambda_0$, resp. $\lambda_1$, the restriction
of the measure $\Leb_\gamma$ to $\gamma(n,\delta)$, resp. $\gamma\setminus \gamma(n,\delta)$, so that $\Leb_\gamma=\lambda_0+\lambda_1$.
Setting $\mu_{\gamma,n}^0=  \frac{1}{n}\,\sum_ {j=0}^{n-1} \Phi^j_\ast \lambda_0$ and
$\mu_{\gamma,n}^1=  \frac{1}{n}\,\sum_ {j=0}^{n-1} \Phi^j_\ast \lambda_1$, we also have
$\mu_{\gamma,n}=\mu_{\gamma,n}^0+\mu_{\gamma,n}^1$.
Hence there exists weak-$\ast$ sublimits
$\mu_0$ of $\mu_{\gamma,n}^0$ and 
$\mu_1$ of $\mu_{\gamma,n}^1$, respectively, such that $\mu=\mu_0+\mu_1$. By construction $\mu_0$ has total mass $\geq 9/10$ while $\mu_1$ has total mass $\leq 1/10$.

Given a set $\gamma_0 \subset M$, let us say that
$z\in M$ is a $\gamma_0$-limit point if there exists a sequence
 of points $x_k\in \gamma_0$ and a sequence of times $n_k \to +\infty$
such that   $ z = \lim_{k\to+\infty} \Psi^{n_k} (x_k) $.
With this terminology, any point $z\in\supp(\mu)$ is a $\gamma$-limit point, and any point $z\in\supp(\mu_0)$ is a $\gamma(n,\delta)$-limit  point.  
Notice that $\mu(\supp(\mu_0))\geq \mu_0(\supp(\mu_0))\geq 9/10$. Therefore
$\mu(\supp(\mu_0)\cap D^0_l)>4/5$.
Hence, if $z\in \supp(\mu_0)\cap D^0_l$ and $\Gamma$ is an $h$-curve $\Gamma\subset W^u_l(z)$ of size $\delta$ around $z$, then $\Gamma$ is accumulated by 
forward iterates $\gamma_n$ of $\gamma$ with $\len{\gamma_n}\geq \delta$.

Applying Lemma
~\ref{lemma gamma cap B mu >0 }
to $\Gamma$, we have $\Leb_\Gamma\left( B(\mu_c)\right)>0$ for some ergodic SRB measure $\mu_c$, which is an ergodic component of a $\Gamma$-limit measure. 

Since $\Gamma$ is accumulated by forward iterates $\gamma_n$ of $\gamma$ with $\len{\gamma_n} \geq \delta$, the Lipschitz continuity of the stable holonomy (see Lemma~\ref{Lipschitz holomomy}) implies that $\Leb_{\gamma_n}\left( B(\mu_c)\right)>0$, and so $\Leb_\gamma\left( B(\mu_c)\right)>0$ because the restriction of $\Phi^{-n}$ along $\gamma_n$ is affine.
\end{proof}

\begin{proof}[Proof of Theorem~\ref{thm  union of basins}]
Let $\mu_1,\ldots, \mu_r$ be the ergodic SRB measures of $\Phi$. Define 
$$B= B(\mu_1)\cup \cdots \cup B(\mu_r) \quad\text{and}\quad A=M\setminus B. $$ 
Assume that $\Leb(A)>0$. We will derive a contradiction from this assumption. 

By Fubini's Theorem, there is an $ h $-curve $\gamma$ such that $\len{\gamma\cap A}>0$. By Lemma~\ref{lemma size delta cc}, for some small $\delta>0$ and some $n_0\in\Nn$, we have $\len{\gamma(n,\delta)\cap A}>0$ for all $n\geq n_0$.

Take a Lebesgue density point $x\in \gamma(n,\delta)\cap A$, and consider a strictly increasing sequence $\{n_i\} $ such that the sequence of $h$-curves $\{ R^{n_i}_\gamma(x)\} $ converges to some $h$-curve $\Gamma$ of length $\geq \delta$.
Notice that the space of all $h$-curves with $\len{\Gamma}\geq \delta$ endowed with the Hausdorff distance is compact.

By Lemma~\ref{len gamma cap B(mu) >0}, $\Leb_\Gamma(B)>0$. Hence, Lemma~\ref{lemma gamma D^+} gives an $l\in\Nn$ such that
$$
\eta:= \len{\Gamma\cap \hat{D}^+_{l}  
\cap B} >0.
$$
Since the local stable curves of points of $ \Gamma \cap \hat{D}^+_{l} \cap B$ have uniform size $\delta_l$, they must intersect $R^{n_i}_\gamma(x)$ for all $  i $ sufficiently large. The Lipschitz continuity of the stable holonomy (see Lemma~\ref{Lipschitz holomomy}) implies
$$\len{ R^{n_i}_\gamma(x) \cap
B } \geq
\frac{\eta}{2}\, \len{R_\gamma^{n_i}(x)}\;.$$
Consider the curve $\gamma':= \Phi^{-n_i}(R^{n_i}_\gamma(x))$. Since the restriction of $\Phi^{n_i}$ to $\gamma'$ is affine, we conclude
$$ \len{ \gamma' \cap B } \geq \frac{\eta}{2}\,\len{\gamma'}\;. $$

However, $x$ is a density point of $A\cap \gamma$, and so the proportion 
$$\frac{\len{\gamma'\cap B}}{\len{\gamma'}}$$ can be made arbitrarily small by choosing a sufficiently small $h$-curve $\gamma'$ around $x$. This fact contradicts the previous conclusion. 
\end{proof}

%% file: global.bbl
\begin{thebibliography}{99}



\bibitem{arroyo09}
A.~Arroyo, R.~Markarian\ and D. P. Sanders, Bifurcations of periodic and chaotic attractors in pinball billiards with focusing boundaries, Nonlinearity {\bf 22} (2009), no.~7, 1499--1522.

\bibitem{arroyo12}
A.~ Arroyo, R.~ Markarian, and D.~P.~Sanders, Structure and evolution of strange attractors in non-elastic triangular billiards, Chaos {\bf 22}, 2012, 026107.

\bibitem{bovi}
C. Bonatti\ and\ M. Viana, SRB measures for partially hyperbolic systems whose central direction is mostly contracting, Israel J. Math. {\bf 115} (2000), 157--193. 

\bibitem{C}
N. Chernov, Decay of correlations and dispersing billiards, J. Statist. Phys. {\bf 94} (1999), no.~3-4, 513--556.

\bibitem{CM} 
N. Chernov\ and\ R. Markarian, {\it Chaotic billiards}, Mathematical Surveys and Monographs, 127, Amer. Math. Soc., Providence, RI, 2006.

\bibitem{CZ} 
N. Chernov\ and\ H.-K. Zhang, Billiards with polynomial mixing rates, Nonlinearity {\bf 18} (2005), no.~4, 1527--1553. 

\bibitem{CZ2009} 
N. Chernov\ and\ H.-K. Zhang, On statistical properties of hyperbolic systems with singularities, J. Stat. Phys. {\bf 136} (2009), no.~4, 615--642.

\bibitem{MDDGP12}
G.~Del Magno, J.~Lopes Dias, P.~Duarte, J.~P.~Gaiv\~ao\ and D.~Pinheiro, Chaos in the square billiard with a modified reflection law, Chaos {\bf 22}, (2012), 026106.

\bibitem{MDDGP13}
G.~Del Magno, J.~Lopes Dias, P.~Duarte, J.~P.~Gaiv\~ao and D.~Pinheiro, SRB measures for polygonal billiards with contracting reflection laws, Comm. Math. Phys. {\bf 329} (2014), 687--723.

\bibitem{MDDGP14}
G.~Del Magno, J.~Lopes Dias, P.~Duarte and J.~P.~Gaiv\~ao, Ergodicity of polygonal slap maps, Nonlinearity {\bf 27} (2014), 1969--1983. 

\bibitem{MDDGP15}
G.~Del Magno, J.~Lopes Dias, P.~Duarte, J.~P.~Gaiv\~ao, Ergodic properties of polygonal billiards with strongly contracting reflection laws, submitted


\bibitem{KS86}
A.~Katok\ et al., {\it Invariant manifolds, entropy and billiards; smooth maps with singularities}, Lecture Notes in Mathematics, 1222, Springer, Berlin, 1986.

\bibitem{markarian10}
R.~Markarian, E.~J.~Pujals\ and\ M.~Sambarino, Pinball billiards with dominated splitting, Ergodic Theory Dynam. Systems {\bf 30} (2010), no.~6, 1757--1786.

\bibitem{Palis}
J.~Palis, A Global View of Dynamics and a Conjecture on the Denseness of Finitude of Attractors, Ast\'{e}risque, {\bf 261} (2000), 339--351.

\bibitem{pe}
Ya.\ B. Pesin, Dynamical systems with generalized hyperbolic attractors: hyperbolic, ergodic and topological properties, Ergodic Theory Dynam. Systems {\bf 12} (1992), no.~1, 123--151. 

\bibitem{Sataev92}
E. A. Sataev, Invariant measures for hyperbolic mappings with singularities, Uspekhi Mat. Nauk {\bf 47} (1992), no. 1(283), 147--202, 240; translation in Russian Math. Surveys {\bf 47} (1992), no.~1, 191--251.



\end{thebibliography}
